\documentclass[11pt]{amsart}
\usepackage{amsmath,amssymb,amsthm,graphicx,verbatim}
\usepackage{tikz}
\usepackage[margin=1in]{geometry}

\newtheorem{lemma}{Lemma}[section]
\newtheorem{corollary}[lemma]{Corollary}
\newtheorem{proposition}[lemma]{Proposition}
\newtheorem{theorem}[lemma]{Theorem}

\theoremstyle{definition}
\newtheorem{definition}[lemma]{Definition}
\theoremstyle{remark}
\newtheorem{remark}[lemma]{Remark}
\newtheorem{example}[lemma]{Example}

\title{Generic spherical unitary dual for Chevalley groups over local fields}
\author{Dan Ciubotaru}
 \address{Mathematical Institute, University of Oxford, Oxford OX2 6GG, UK}
        \email{dan.ciubotaru@maths.ox.ac.uk}

\date{}

\begin{document}

\begin{abstract}
We prove that a generic spherical parameter for the graded affine Hecke algebra with equal parameters is unitary if and only if its normalized intertwining
forms are positive on the reflection representation and on the
irreducible constituents of its second symmetric power.  The proof
combines a signature formula for the reflection representation, Yu's
uniform two-wall operation, a simply-laced Jantzen recursion, and a
root-poset analysis.  In types
\(B_n\) and \(C_n\), we prove more sharply that the reflection
representation and the traceless-diagonal constituent of
\(\operatorname{Sym}^2V\) suffice. We also give a shorter 
second proof when arbitrary Weyl-group types are allowed.   The consequence of the first proof
(via Barbasch--Vogan petite \(K\)-types) is that the generic spherical
unitary dual of a Chevalley group over the real numbers or a
nonarchimedean local field, and for simply-laced groups and complex
symplectic groups also over the complex numbers, is independent of the
field.  Moreover, the answer
has a simple uniform description as the Weyl conjugates of a disjoint
union of \(2^{\nu(R^\vee_{\mathrm{sh}})}\) alcoves in the affine Weyl
group arrangement in the fundamental Weyl chamber, where
\(\nu(R^\vee_{\mathrm{sh}})\) is the matching number of the Dynkin
diagram of the subsystem of short coroots.
\end{abstract}

\maketitle

\section*{Introduction}
\addcontentsline{toc}{section}{Introduction}

Let \(\mathbb H(R)\) be the graded affine Hecke algebra attached to a
reduced crystallographic root system \(R\), with all parameters equal to
one.  A real unramified parameter \(\nu\) determines a spherical standard
module and normalized Hermitian intertwining forms on the multiplicity
spaces of the finite Weyl-group types.  The module is unitary precisely
when all these forms are positive.  Although this criterion is finite, it
appears at first to require the complete character theory of the Weyl
group.  The purpose of this paper is to show that, for generic spherical
parameters, only the first two symmetric degrees of the reflection
representation are needed.

Here are the principal results in a form that also records the resulting
classification.

\medskip

\noindent\textbf{Theorem A (uniform detection and classification).}
\emph{Let \(R\) be a reduced crystallographic root system and let
\(\nu\) be a generic Hermitian spherical parameter for the
equal-parameter graded affine Hecke algebra \(\mathbb H(R)\).  The
spherical module with parameter \(\nu\) is unitary if and only if its
normalized forms are positive definite on the reflection representation
\(V\) and on every irreducible constituent of
\(\operatorname{Sym}^2V\).  Equivalently, the root ideal of its
dominant representative is a Yu ideal.  In a fundamental chamber the
unitary set is a disjoint union of
\[
                         2^{k(R)}
\]
affine alcoves (relative to the Hermitian locus when
\(w_0\ne-1\)), where \(k(R)\) is the matching number of the Dynkin
graph of short coroots.}

\medskip

\noindent\textbf{Theorem B (sharper \(B/C\) criterion).}
\emph{For \(R=B_n\) or \(C_n\), it is enough to test \(V\) and the single
irreducible traceless-diagonal constituent
\[
 U_n=\left\{\sum_{i=1}^n c_i e_i^2:\sum_{i=1}^n c_i=0\right\}
       \subset\operatorname{Sym}^2V.
\]
Thus positivity on \(V\) and \(U_n\) already characterizes the Yu
cells and spherical unitarity.}

\medskip

\noindent\textbf{Corollary C (field independence).}
\emph{After the standard normalization of real unramified parameters,
the generic spherical unitary dual of a Chevalley group has the same
description over \(\mathbb R\) and over every nonarchimedean local
field.  The same description holds over \(\mathbb C\) for simply-laced
groups and for \(\operatorname{Sp}(2n,\mathbb C)\).}

\medskip

The restriction to the Weyl types in Theorem A is essential for the
application.  Under the Barbasch--Vogan philosophy of petite
\(K\)-types \cite{BarbaschPetite}, their normalized operators can be
transported uniformly to spherical principal series of split real
groups; for complex groups the analogous role is played by zero-weight
spaces of small representations.  The nonarchimedean comparison comes
from the Iwahori and graded-Hecke correspondences
\cite{BM,Lusztig}.  In type \(B/C\), Theorem B is stronger than
Theorem A and is precisely what adds the complex symplectic groups.

The set of unitary cells itself was known from the case-by-case
classifications \cite{BM,BarbaschPetite,BaClassical,BaCi,CiF4}.
What is new here is not another enumeration of those cells.  The paper
proves uniformly that Yu's cells exhaust the generic spherical unitary
dual; it gives the uniform detection theorem using only \(V\) and
\(\operatorname{Sym}^2V\); it identifies Yu's exponent \(\ell(R)\) with
a matching number; and in types \(B_n,C_n\) it replaces the full
symmetric-square test by the single constituent \(U_n\).  A second,
shorter proof using boundary-induced Weyl types is also uniform and does
not invoke the case-by-case lists of relevant types.

For $\mathbb H(R)$, the reducibility hyperplanes are
\[
                       \beta(\nu)=1,\qquad \beta\in R^+.
\]
They divide the dominant chamber into cells, encoded by the upper root
ideals
\[
                 I(\nu)=\{\beta\in R^+:\beta(\nu)>1\}.
\]
Yu discovered a uniform operation crossing two consecutive walls in an
\(A_2\)-configuration \cite{Yu}.  Positivity on every Weyl-group type is equivalent
at the two outer cells of such a move.  Starting at the fundamental cell
therefore produces a distinguished collection of unitary cells, which we
call the {\it Yu cells}.  Yu computed their number and observed, using the
classifications available at the time, that they exhaust the generic spherical
unitary dual.  The missing point was a uniform proof that every other
cell is nonunitary.

The first proof supplied here retains the restriction to symmetric degree
two.  We first establish an alternating-height formula for the signature
on \(V\).  Reflection positivity becomes the elementary condition
\[
             \sum_{\beta\in I}(-1)^{\operatorname{ht}(\beta)+1}=0.
\]
We then repeatedly remove Yu pairs until the ideal is {\it locked}, see Definition \ref{d:locked}.  In
simply-laced type, a one-wall Jantzen calculation identifies the first
form on the radical of the symmetric-square intertwiner with a reflection
form for an orthogonal root subsystem.  This gives a recursive signature
formula, Corollary \ref{cor:sym2-signature}.  The remaining task is combinatorial: classify the nonempty
balanced locked ideals and prove that the recursive signature is nonzero
on each of them.  Even type \(D\) is the uniform classical model; via a standard 
Levi subgroup argument the conclusion also follows for types \(A\) and odd \(D\); the
exceptional simply-laced cases come down to a finite verification in
\(E_8\).  In types \(B\) and \(C\), a cofactor and an exact
maximal-parabolic matrix coefficient show more sharply that the reflection
representation together with the traceless-diagonal constituent \(U_n\)
already detects unitarity.  Exact finite checks handle \(F_4\) and \(G_2\).

\medskip

We also give a shorter second proof when arbitrary Weyl-group types are
allowed.  Parabolic induction of Hermitian forms transports a
nonpositivity certificate from a Levi subsystem to the ambient root
system.  This reduces the analytic calculation to the explicit basic
cases \(B_2\), \(D_4\), and \(G_2\); the rest is a root-poset extraction of one
of these basic cases from every nonempty locked ideal.  This argument is more
economical, but the first proof is retained because its symmetric-degree-two restriction is what is needed for the petite-\(K\)-type application.

\medskip

There is also a simple interpretation of Yu's power of two.  Let
\(R^\vee_{\rm sh}\) be the short-coroot subsystem and let
\(\Gamma(R^\vee_{\rm sh})\) be its Dynkin graph.  If \(\ell\) denotes
the matching number of this tree (or forest), then the number of Yu cells is
\(2^\ell\).  Equivalently, \(\ell\) is the largest number of mutually
disjoint standard-parabolic \(A_2\)-factors in the short-coroot subsystem.
Using the adjacency spectrum of a simply-laced Dynkin diagram, it is also
\[
 \ell=\sum_{\Phi\subset R^\vee_{\rm sh}}
          \#\{e\in\operatorname{Exp}(\Phi):e<h_\Phi/2\},
\]
where $h_\Phi$ is the Coxeter number.
Thus both the geometry of the unitary set and the exponent in its
cardinality are determined directly by the root system.

The paper begins by recalling the normalized intertwiners and Yu's
two-wall theorem, followed by the interpretation of \(\ell\).  We then
prove the reflection signature formula and the uniform reduction to
locked ideals.  The simply-laced and non-simply-laced symmetric-square
arguments occupy the middle sections.  The boundary-induced proof is
given separately near the end, followed by the graded-Hecke theorem and
its local-field consequence.

\section{Spherical intertwiners, cells, and Yu's theorem}
\label{sec:yu-summary}

Let \(R\subset V\) be a reduced crystallographic root system ($V$ a finite-dimensional real vector space), viewed as
a coroot system in the group-theoretic applications, let \(W\) be its
Weyl group, and fix simple roots \(\Delta\) and the corresponding
dominant chamber \(C\subset V\).  We use the equal-parameter graded affine Hecke
algebra normalization in which the reducibility hyperplanes
in \(C\) are
\[
                         H_\beta=\{\nu\in V:\beta(\nu)=1\},
                         \qquad \beta\in R^+ .               \tag{1.1}
\]
This $\nu$ is the usual real form of the unramified parameter under the Satake
isomorphism.  Throughout the paper \emph{generic} means regular and off
all the hyperplanes in (1.1).  In the closed dominant chamber this is
equivalent to lying in its interior and satisfying
\(\beta(\nu)\ne1\) for every positive root.
When \(w_0\ne-1\), Hermitian parameters are treated on the Hermitian Levi
flats of Lemma~\ref{lem:boundary-hermitian-levi}; cells and alcoves in
that case are understood relative to those flats.

Suppose first that \(w_0=-1\), and choose a reduced expression
\(w_0=s_N\cdots s_1\).  If \(s_i=s_{\alpha_i}\), set
\[
 \beta_i=s_1\cdots s_{i-1}(\alpha_i).
\]
The \(\beta_i\)'s are the positive roots, each occurring once.  For a
unitary \(W\)-module \((\sigma,E)\), the normalized long-intertwining form
is represented, up to a positive scalar, by
\[
 A_E(\nu)=
 \frac{1+\beta_N(\nu)\sigma(s_N)}{1+\beta_N(\nu)}\cdots
 \frac{1+\beta_1(\nu)\sigma(s_1)}{1+\beta_1(\nu)}.         \tag{1.2}
\]
The product is independent of the reduced expression after the standard
identifications of source and target.  Since \(w_0\nu=-\nu\), it is
Hermitian for real \(\nu\).  Its restriction to each \(W\)-isotypic
multiplicity space is the invariant Hermitian form on the corresponding
type in the spherical module.

The Barbasch--Moy unitarity criterion~\cite{BM} says that a generic spherical
module is unitary exactly when these forms are positive definite for all
irreducible \(W\)-types.  Our theorem shows that the reflection type and
the constituents of its symmetric square already suffice.

The hyperplanes (1.1) cut the interior of \(C\) into cells.  The cell of
\(\nu\) is encoded by
\[
 I(\nu)=\{\beta\in R^+:\beta(\nu)>1\}.                      \tag{1.3}
\]
This is an upper ideal in the root poset; its minimal elements form an
antichain, and every realizable antichain gives a cell.  The fundamental
cell \(U_0\), containing the origin in its closure, corresponds to the
empty ideal.

Color the positive roots by
\[
        \epsilon(\beta)=(-1)^{\operatorname{ht}(\beta)+1}.
\]
An upper ideal \(I\subset R^+\) is called \emph{balanced} if
\[
                    \sum_{\beta\in I}\epsilon(\beta)=0.
\]

Let \(I=I(U)\) be the root ideal of a cell \(U\).  A root
\(\beta\in I\) is called \emph{exposed} if \(H_\beta\) is a wall of
\(U\) and crossing this wall toward the fundamental cell deletes
\(\beta\) from the root ideal.  Thus the adjacent cell has root ideal
\(I\setminus\{\beta\}\).  In particular, an exposed root is a minimal
element of \(I\).

We now isolate the result of Yu~\cite{Yu} used in this paper.  Let
\(U,U',U''\) be
three successive cells, with \(U,U'\) separated by \(H_p\) and
\(U',U''\) separated by \(H_{p'}\), oriented so that
\[
 p(U)<1<p(U'),\qquad p'(U')<1<p'(U''),                      \tag{1.4}
\]
and suppose
\[
             \langle p,(p')^\vee\rangle
             =\langle p',p^\vee\rangle=1.                 \tag{1.5}
\]

In the root ideal, passage from \(U''\) to \(U\) first deletes the
exposed root \(p'\) and then the exposed root \(p\); equivalently,
\(p'\) is exposed in \(I(U'')\) and \(p\) is exposed in
\(I(U'')\setminus\{p'\}\).  We call this a \emph{Yu move}.

\begin{theorem}[Yu's two-wall theorem]\label{thm:yu-two-wall}
For every unitary \(W\)-module \(E\),
\[
 A_E|_U>0\quad\Longleftrightarrow\quad A_E|_{U''}>0.        \tag{1.6}
\]
Consequently every cell obtained from \(U_0\) by a sequence of inverse
Yu moves is positive on every \(W\)-type and hence is spherical unitary.
\end{theorem}

\begin{proof}[Summary of Yu's argument]
Yu first chooses a simple system containing \(-p,p'\).  Equivalently one
may choose a reduced expression for \(w_0\) in which the two root factors
belonging to \(p,p'\) are consecutive.  This is his root-system lemma;
the exceptional rank-two case \(G_2\) is checked directly.  At a generic
point of either wall, if \(-1\) has multiplicity \(m\) in the relevant
reflection on \(E\), the determinant of (1.2) has order \(m\).  Hermitian
analytic perturbation shows that exactly \(m\) eigenvalues cross zero.

At the common codimension-two face the singular part is
\[
                         (1+s_{p'})(1+s_p).                 \tag{1.7}
\]
Condition (1.5) makes \(\langle s_p,s_{p'}\rangle\cong S_3\).  On the
three irreducible \(S_3\)-modules--trivial, sign, and reflection--one
checks respectively that the ranks of (1.7) are \(1,0,1\), equal to the
ranks of \(1+s_p\).  Thus the same \(m\)-dimensional cluster crosses at
the two walls and has the same sign in the two outer cells.  The
remaining eigenvalues do not vanish during this local passage, so their
signs are unchanged.  Hence the form is positive definite in one outer
cell if and only if it is positive definite in the other.  This proves
(1.6).
\end{proof}

\begin{definition}\label{d:locked}
The cells reachable from \(U_0\) by inverse Yu moves are called
\emph{Yu cells}.  An upper ideal is \emph{locked} if no Yu move can be
applied to it.
\end{definition}

Yu's note also gives labelled alcove diagrams, an involutive symmetry of
those diagrams, and the powers-of-two counts
\[
\begin{array}{c|c}
\text{coroot type}&\#\{\text{Yu cells}\}\\ \hline
B_n&1\qquad(n\ge2),\\
C_{2m-1},\ C_{2m},\ D_{2m}&2^{m-1},\\
E_7,E_8&8,16,\\
F_4,G_2&2,2.
\end{array}                                                  \tag{1.8}
\]
It observes, on the basis of the then available classification and
computations, that these cells exhaust the generic spherical unitary
dual, that each is an affine alcove, and that the diagram has a head-tail
symmetry.  The exhaustion is precisely the assertion for which we give a
new proof.  Only the uniform two-wall theorem above, not the case-by-case
classification, is used in our argument.

\subsection{A root-system interpretation of Yu's exponent}
\label{sec:yu-count-interpretation}

The powers of two in (1.8) admit a compact root-theoretic
interpretation.  Let \(R^\vee_{\rm sh}\) be the subsystem of short
coroots, let \(\Gamma_{\rm sh}\) be its Dynkin graph, and let
\(\nu(\Gamma_{\rm sh})\) denote its matching number, that is, the cardinality of the maximum subset of edges in the graph such that no two edges share a common vertex.  Equivalently,
\(\nu(\Gamma_{\rm sh})\) is the largest \(r\) for which a standard
parabolic subsystem of \(R^\vee_{\rm sh}\) has a factor \(A_2^r\).

\begin{proposition}[interpretation of the power of two]
\label{prop:yu-matching-exponent}
Let \(2^{k(R^\vee)}\) be the number of generic Yu cells, with cells
understood relative to the Hermitian Levi flats when \(w_0\ne-1\).  Then
\[
                  k(R^\vee)=\nu(\Gamma_{\rm sh}).          \tag{1.9}
\]
If
\(R^\vee_{\rm sh}=\Phi_1\times\cdots\times\Phi_s\), with exponents
\(e_{ij}\) and Coxeter numbers \(h_i\), then equivalently
\[
 k(R^\vee)
   =\sum_{i=1}^s\#\{j:e_{ij}<h_i/2\}.                     \tag{1.10}
\]
\end{proposition}

\begin{proof}
The short-coroot subsystems and their matching numbers are
\[
\begin{array}{c|c|c}
R^\vee&R^\vee_{\rm sh}&\nu(\Gamma_{\rm sh})\\ \hline
A_n&A_n&\lfloor n/2\rfloor,\\
B_n&A_1^n&0,\\
C_n&D_n&\lfloor(n-1)/2\rfloor,\\
D_n&D_n&\lfloor(n-1)/2\rfloor,\\
E_6&E_6&3,\\
E_7&E_7&3,\\
E_8&E_8&4,\\
F_4&D_4&1,\\
G_2&A_2&1.
\end{array}                                                \tag{1.11}
\]
For \(D_n\), a matching uses at most one edge at the fork; deleting
the two vertices of that edge leaves a path and gives
\[
 \nu(D_n)=1+\left\lfloor\frac{n-3}{2}\right\rfloor
          =\left\lfloor\frac{n-1}{2}\right\rfloor .       \tag{1.12}
\]
The remaining entries are immediate from the diagrams.  For the systems
with \(w_0=-1\), comparison with (1.8) proves (1.9).  Types \(A\), odd
\(D\), and \(E_6\) follow by the Hermitian Levi reduction of
Lemma~\ref{lem:boundary-hermitian-levi}: matching number is additive on
the resulting Dynkin graph, just as the number of cells multiplies over
the irreducible factors.  This proves (1.9) in all types.

We prove (1.10) uniformly.  Let \(\Phi\) be one simply-laced irreducible
component, let \(A_\Phi\) be the adjacency matrix of its Dynkin tree,
and write \(n=\operatorname{rank}\Phi\).  For a forest, the nullity of
the adjacency matrix is the number of vertices left uncovered by a
maximum matching \cite{Meszaros}; hence
\[
             \operatorname{rank}A_\Phi=2\nu(\Gamma_\Phi). \tag{1.13}
\]
For completeness, this also follows immediately by expanding
\(\det(tI-A_\Phi)\): because the graph has no cycles, a nonzero
permutation term is a matching, and the least power of \(t\) is
\(t^{n-2\nu(\Gamma_\Phi)}\).

The standard Coxeter spectral formula \cite{Humphreys} gives the
adjacency eigenvalues
\[
                    2\cos\frac{\pi e_j}{h}.                \tag{1.14}
\]
The tree is bipartite, so its nonzero eigenvalues occur in opposite
pairs.  Equations (1.13)--(1.14) therefore give
\[
 \nu(\Gamma_\Phi)
  =\#\{\lambda(A_\Phi)>0\}
  =\#\{j:e_j<h/2\}.                                       \tag{1.15}
\]
Summing over the components proves (1.10).
\end{proof}

\begin{remark}
Formula (1.9) says that Yu's exponent is the maximal number of
vertex-disjoint short-coroot \(A_2\)-dominoes.  It does not assert that
the Yu-move graph is a hypercube.  Already in the first classical case
with \(k=2\), that graph is a path on four vertices rather than a square.
The power of two is therefore better viewed as arising from successive
rank-two doubling than from \(k\) globally commuting moves.
\end{remark}

\section{The reflection signature formula}
\label{sec:reflection-signature}

We now incorporate the reflection-signature argument. Assume in this
section that the longest Weyl element acts by \(-1\) on \(V\).
For a nondegenerate Hermitian form \(B\), write \(n_-(B)\) for the
number of its negative eigenvalues.

Give \(R^+\) its usual graded root-poset order and write
\(\operatorname{ht}(\beta)\) for the height.  For a plane \(P\subset V\),
the roots \(R\cap P\) form a (possibly reducible) rank-two root system,
and \(R^+\cap P\) is a positive system for it.  Moreover, order in this
rank-two positive system implies order in \(R^+\).  Indeed, a positive
linear combination of its simple roots is also a positive linear
combination of the original simple roots.

\begin{lemma}[rank-two localization]\label{lem:reflection-rank-two}
Let \(\beta_1,\beta_2\in R^+\).
\begin{enumerate}
\item If they are incomparable, a generic point of
\(H_{\beta_1}\cap H_{\beta_2}\) has
\(\dim\ker A_V=2\).
\item If \(\beta_2\) covers \(\beta_1\) in the root poset, a generic point
of their intersection has \(\dim\ker A_V=1\).
\end{enumerate}
In the first case the two transverse first Jantzen directions are
independent.  In the second they induce opposite orientations on the
common radical line.
\end{lemma}

\begin{proof}
Let \(P=\mathbb R\beta_1+\mathbb R\beta_2\), and choose the simple roots
\(\beta'_1,\beta'_2\) of \(R^+\cap P\).  They are incomparable.  The
standard antichain lemma for Weyl groups (equivalently, the two-element
case of Sommers' lemma~\cite{Sommers}) gives \(y\in W\) carrying them to simple roots
\(\gamma_1,\gamma_2\).  If \(w_P\) is the longest element of the dihedral
subgroup they generate, factor
\[
                              w_0=xw_Py.                    \tag{2.0}
\]
After braid cancellations, the root-intertwiner factors whose roots lie
in \(P\) form exactly the \(w_P\)-block; every other factor is invertible
at a point of \(H_{\beta_1}\cap H_{\beta_2}\) avoiding the remaining
root walls.  Such points exist because a finite union of proper affine
subspaces cannot cover that intersection.  Hence the kernel and its
first transverse form are those of the rank-two block.

There are only \(A_1\times A_1,A_2,B_2,G_2\).  Multiplication of their
two-dimensional reflection matrices gives kernel dimensions two for an
incomparable pair and one for a cover pair.  Symmetric elimination gives
independent transverse linear terms in the first case and opposite
linear terms on the common kernel in the second.  This proves all four
claims.  Notice that the second simple root in the localization is
\(\gamma_2\); this corrects the harmless repeated \(\gamma_1\) in the
statement of Lemma 5.5 of the original note.
\end{proof}

\begin{theorem}[reflection signature \cite{Crisan}]
\label{thm:reflection-signature}
For a regular dominant parameter \(\nu\), put
\(I(\nu)=\{\beta\in R^+:\beta(\nu)>1\}\).  Then
\[
 n_-\bigl(A_V(\nu)\bigr)
   =\sum_{\beta\in I(\nu)}(-1)^{\operatorname{ht}(\beta)+1}. \tag{2.1}
\]
In particular, the reflection form is positive definite if and only if
\(I(\nu)\) is balanced.
\end{theorem}

\begin{proof}
By Section~\ref{sec:yu-summary}, the normalized long intertwiner is a
real-analytic Hermitian matrix away
from its scalar poles.  On a generic point of \(H_\beta:\beta(\nu)=1\),
display the factor belonging to \(\beta\) in a reduced expression.  All
other factors are invertible, and the \((-1)\)-space of \(s_\beta\) on
\(V\) is one-dimensional.  Hence the form has a one-dimensional radical.
The determinant formula has a simple zero there, so the corresponding
eigenvalue has a nonzero first derivative.  Thus crossing \(H_\beta\)
changes the negative index by a sign \(j(\beta)=\pm1\).

It remains to determine that sign and to see that it is independent of
the generic point of the wall.  Lemma~\ref{lem:reflection-rank-two} gives
the following two facts:
\begin{enumerate}
\item if the two roots are incomparable, the form has a
two-dimensional radical and the two transverse first Jantzen forms are
independent; moving around the codimension-two intersection does not
change either wall orientation;
\item if \(\beta'\) covers \(\beta\) in the root poset, the two walls
have the same radical line at their intersection and their transverse
orientations are opposite.
\end{enumerate}
Real-analytic Hermitian perturbation justifies following the vanishing
eigenvalues and the rank-one determinant factor makes every generic
crossing transverse.

The generic pieces of a fixed wall are connected through incomparable
rank-two crossings, so the first fact makes \(j(\beta)\) well-defined.
For a simple root the crossing adjacent to the fundamental cell changes
one positive eigenvalue to a negative one, hence \(j(\alpha)=1\).  The
second fact reverses the sign along every root-poset cover.  Induction on
height therefore gives
\[
                         j(\beta)=(-1)^{\operatorname{ht}(\beta)+1}.
\]
Starting with the positive form at the origin and summing the wall jumps
along \(t\nu\) proves (2.1).
\end{proof}

\begin{remark}
The assumption \(w_0=-1\) is sufficient for the systems where the formula
is used directly below: even \(D\), \(E_7,E_8\), \(B,C,F_4,G_2\).  Types
\(A\), odd \(D\), and \(E_6\) are obtained by Levi inheritance, so no
extension of this proof to \(w_0\ne-1\) is needed.
\end{remark}

\section{The uniform reduction principle}
\label{sec:uniform-reduction}

Let $I$ denote a root ideal in $R$, as before.

\begin{proposition}[uniform reduction principle]
\label{prop:uniform-reduction}
Assume the following three assertions for an irreducible root system
\(R\):
\begin{enumerate}
\item the reflection form is positive on the cell of \(I\) if and only if
      \(I\) is balanced;
\item positivity on every fixed \(W\)-type is equivalent at the two ends
      of a Yu move;
\item every nonempty balanced locked ideal fails positivity on at least
      one irreducible constituent of \(\operatorname{Sym}^2V\).
\end{enumerate}
Then the reflection and symmetric-square forms are simultaneously
positive exactly on Yu's cells.  Consequently they characterize generic
spherical unitarity.
\end{proposition}

\begin{proof}
Suppose first that the reflection and all symmetric-square constituent
forms are positive.  By (1), \(I(\nu)\) is balanced.  Repeatedly apply Yu
moves.  The process terminates because the cardinality decreases by two.
By (2), all the test forms remain positive.  The terminal ideal is
balanced and locked, so (3) forces it to be empty.  Thus the original
cell is connected to the fundamental cell by Yu moves, which is the
definition of a Yu cell.  Yu's theorem makes it unitary.  Conversely, a
unitary module has positive form on every finite Weyl-group type.
\end{proof}

\begin{remark}
One might ask whether the reduction is \emph{confluent}: if two
different Yu moves can be applied to the same ideal, must the two
resulting reduction paths eventually reach the same locked ideal?  We do
not need such a statement.  Starting from a balanced ideal
\(I\), choose any available Yu move and continue until a locked ideal
\(K\) is reached.  The process terminates because every Yu move removes
two roots.  Positivity of each test form is preserved in both directions
across every Yu move.  Consequently, if all the test forms are positive
on \(I\), they remain positive on \(K\).  But every nonempty balanced
locked ideal is detected by one of the test forms, in the sense that
this form is not positive definite there.  It follows that \(K\) must
be the empty ideal.  This argument applies to every possible sequence
of Yu moves, even without knowing in advance that different sequences
have the same endpoint.

Likewise, we do not need equality of the complete signatures of the
Hermitian forms at the two ends of a Yu move.  On a generic cell the
form is nondegenerate, and its complete signature is the pair
\((n_+,n_-)\) of numbers of positive and negative eigenvalues.  These
numbers may change under a Yu move.  The proof uses only the weaker
equivalence
\[
          A_E|_U>0\quad\Longleftrightarrow\quad A_E|_{U''}>0
\]
for each test \(W\)-type \(E\).
\end{remark}

\section{The one-wall recursion for the symmetric square}

Let \(R\subset V\) be a simply-laced reduced root system with Weyl group
\(W\).  Assume that the reflection-signature theorem applies to every
orthogonal subsystem \(R\cap\beta^\perp\) occurring below.  This holds in
the two cases used later: \(D_{2m}\), where the subsystem is
\(A_1\times D_{2m-2}\), and \(E_8\), where it is \(E_7\).  Normalize
parameters so that the reducibility walls are \(\alpha(\nu)=1\). For a
real \(W\)-module \(E\), let \(A_E(\nu)\) be the
normalized Hermitian long-intertwining form. Its scalar denominators are
positive in the dominant chamber.

For \(\beta\in R^+\), put
\[
 R_\beta=R\cap\beta^\perp,\qquad V_\beta=\beta^\perp .
\]
The factor of \(V_\beta\) orthogonal to the span of \(R_\beta\) is retained
as a trivial \(W(R_\beta)\)-module. On the wall \(\beta(\nu)=1\), let
\(\bar\nu_\beta\) be the restricted parameter:
\(\gamma(\bar\nu_\beta)=\gamma(\nu)\) for \(\gamma\in R_\beta\).
We write
\(v\odot w=v\otimes w+w\otimes v\) for the symmetric tensor product;
its normalization will play no role.
We write \(\gamma\lessdot\beta\) when \(\beta\) covers \(\gamma\) in
the positive-root poset.

\begin{lemma}[codimension-two rank table]\label{lem:rank-two-table}
Let two reducibility walls meet, and let \(Q_\beta\) denote the first
Jantzen quotient form associated with the \(\beta\)-wall. Locally:
\[
\begin{array}{c|c|c}
(\beta,\gamma)&\text{rank behaviour on }\operatorname{Sym}^2V
 &\text{effect on }Q_\beta\\ \hline
0&\text{two commuting }A_1\text{ factors}
 &\text{one reflection wall in }R_\beta\\
-1&A_2,\ \beta,\gamma\text{ incomparable}
 &\text{no degeneracy and no signature change}\\
1&A_2,\ \beta-\gamma\text{ a root}
 &\text{common radical; opposite wall orientations.}
\end{array}
\]
Thus the only interior walls of \(Q_\beta\) are the walls of \(R_\beta\).
\end{lemma}

\begin{proof}
The orthogonal complement of the plane generated by the two roots only
adds trivial summands, so the essential calculation is in rank (2). In the orthogonal
case write
\[
 V=\mathbb R\beta\oplus\mathbb R\gamma\oplus T
\]
and inspect the four sign characters in the symmetric square. In the
nonorthogonal cases the group is \(S_3\). If \(U\) is its two-dimensional
reflection representation, then
\[
 \operatorname{Sym}^2(U\oplus T)
 =(\mathbf1\oplus U)\oplus U^{\oplus\dim T}
   \oplus\operatorname{Sym}^2T.
\]
On \(\mathbf1\) and \(U\), the ranks of
\((1+s_2)(1+s_1)\) are respectively \(1\) and \(1\), the latter equal to
the rank of either single singular factor on \(U\). This proves the rank
claims. The incomparable and comparable placements in the \(A_2\) string
give respectively unchanged and opposite orientations.
\end{proof}

\begin{lemma}[one-wall recursion]\label{lem:one-wall}
Let \(\nu_0\) be a point of the wall
\[
                   H_\beta=\{\nu:\beta(\nu)=1\}
\]
which lies on no other reducibility hyperplane.  Choose a vector
\(\eta\) with \(\beta(\eta)=1\), and cross the wall along
\[
                       \nu(t)=\nu_0+t\eta.
\]
Thus \(t=\beta(\nu(t))-1\), with \(t>0\) on the side
\(\beta(\nu)>1\) and \(t<0\) on the side \(\beta(\nu)<1\).

Via the nonsingular intertwining factors, the radical of
\(A_{\operatorname{Sym}^2V}(\nu_0)\) is identified with
\[
             K_\beta=\beta\odot\beta^\perp\cong V_\beta.
\]
The first Jantzen form, namely the coefficient of \(t\) in the
restriction of
\(A_{\operatorname{Sym}^2V}(\nu(t))\) to \(K_\beta\), is congruent to
\[
 -\epsilon(\beta)c_\beta(\nu_0)
 A^{R_\beta}_{V_\beta}(\bar\nu_\beta),                 \tag{4.1}
\]
where \(c_\beta(\nu_0)>0\) and
\[
             \epsilon(\beta)=(-1)^{\operatorname{ht}(\beta)+1}.
\]
If \(k_\beta\) is the negative index of the reflection form on the
right-hand side of \emph{(4.1)}, then crossing from
\(\beta(\nu)<1\) to \(\beta(\nu)>1\) changes the negative index by
\[
 \epsilon(\beta)\bigl(\dim V_\beta-2k_\beta\bigr).       \tag{4.2}
\]

\end{lemma}

\begin{proof}
Choose a reduced expression in which the factor associated with
\(\beta\) is displayed, and evaluate the resulting product along the
path \(\nu(t)=\nu_0+t\eta\).  At \(t=0\), the \(\beta\)-factor is
singular, whereas all the other factors are invertible.  It follows
that the radical of the full product is the inverse image, under the
invertible factors, of the \((-1)\)-eigenspace of \(s_\beta\) on
\(\operatorname{Sym}^2V\).  Since
\[
 V=\mathbb R\beta\oplus\beta^\perp,\qquad
 (\operatorname{Sym}^2V)^{-s_\beta}
       =\beta\odot\beta^\perp,
\]
the radical is thereby identified with
\[
                  K_\beta=\beta\odot\beta^\perp.
\]
When the product is differentiated at \(t=0\) and restricted to this
radical, derivatives of the invertible factors make no contribution:
each such term still contains the singular \(\beta\)-factor, which
vanishes on the corresponding radical vector.  Hence the first
Jantzen form is obtained from the derivative of the singular
rank-one factor, conjugated by the remaining invertible factors.

For a simple root \(\beta\), begin with the portion of \(H_\beta\)
adjacent to the fundamental alcove.  There the derivative of the
rank-one factor is negative, while the remaining factors induce the
positive reflection form for \(R_\beta\).  As \(\nu_0\) moves within
the generic part of \(H_\beta\), the first Jantzen form can become
singular only where \(H_\beta\) meets another reducibility wall.
Lemma~\ref{lem:rank-two-table} shows that a nonorthogonal,
incomparable intersection produces no change of signature, whereas an
orthogonal intersection produces exactly the wall crossing of the
reflection form for \(R_\beta\).  This proves \emph{(4.1)} for a simple
root throughout every generic component of its wall.

If \(\beta'\lessdot\beta\) is a root-poset cover, the two affine walls meet
on the chamber wall corresponding to the simple root \(\beta-\beta'\).
The last row of the rank table identifies their common radical and reverses
orientation. Induction along a saturated chain gives
\((-1)^{\operatorname{ht}(\beta)-1}=\epsilon(\beta)\);
gradedness makes this independent of the chain. Sylvester's law turns
equality of signatures into the pointwise congruence \emph{(4.1)}.

Finally, because \(t=\beta(\nu(t))-1\), differentiation of the
singular rank-one factor with respect to \(t\) contributes the
universal minus sign in (4.1).  A nondegenerate form of dimension
\(\dim V_\beta\) and negative index \(k_\beta\) has signature
difference
\[
                         \dim V_\beta-2k_\beta.
\]
Therefore crossing from \(t<0\), that is \(\beta(\nu)<1\), to \(t>0\),
that is \(\beta(\nu)>1\), changes the negative index by (4.2).
\end{proof}

\begin{corollary}[recursive signature formula]\label{cor:sym2-signature}
Perturb \(\nu\) within its cell so that the values \(\beta(\nu)\) are
pairwise distinct. Let \(h_\beta\) be height in
\(R_\beta^+=R^+\cap\beta^\perp\), and put
\[
 k_\beta(\nu)=
 \sum_{\substack{\gamma\in R_\beta^+\\
                  \gamma(\nu)>\beta(\nu)}}
 (-1)^{h_\beta(\gamma)+1}.
\]
Then
\[
 n_-\bigl(A_{\operatorname{Sym}^2V}(\nu)\bigr)
 =\sum_{\substack{\beta\in R^+\\\beta(\nu)>1}}
 (-1)^{\operatorname{ht}(\beta)+1}
 \left(\operatorname{rank}R-1-2k_\beta(\nu)\right).     \tag{4.3}
\]
\end{corollary}

\begin{proof}
Cross the walls successively along \(t\nu\). At the \(\beta\)-wall,
precisely the orthogonal roots \(\gamma\) with
\(\gamma(\nu)>\beta(\nu)\) have already been crossed. The reflection
signature formula in \(R_\beta\) says that their alternating sum is
\(k_\beta(\nu)\). Sum (4.2), starting with the positive form at \(t=0\).
\end{proof}

\begin{remark}
The unmodified statement is false across a multiple bond. For example, in
type \(B_3\), at the wall of a middle short coordinate root, the
symmetric-square jump is \(2\), whereas the equal-parameter reflection form
of the orthogonal \(B_2\) subsystem has signature difference \(0\).
The \(B,C,BC\) families therefore require a separate signed-permutation
argument.
\end{remark}

\section{Locked balanced ideals in \(D_{2m}\)}
\label{sec:d-even-locked}

Put \(n=2m\), and write
\[
 x_{pq}=e_p-e_q,\qquad y_{pq}=e_p+e_q\qquad(1\leq p<q\leq n).
\]
Color a root by
\(\epsilon(\beta)=(-1)^{\operatorname{ht}(\beta)+1}\), and call an
upper ideal {\it locked} if no two successive deletions of minimal roots form
an \(A_2\)-pair.  Let
\[
 T_{2j}=\{\beta\in R^+:\operatorname{ht}(\beta)\geq2j\}.
\]
For \(1\leq j<m\), define the fork staircase
\[
 S_{n,j}=\{y_{pq}:n-2j+1\leq p<q\leq2n-2j-p\},             \tag{5.1}
\]
and put
\[
                         K_{n,j}=T_{2j}\setminus S_{n,j}.    \tag{5.2}
\]
For \(j=1\), the staircase is empty.  Its successive row lengths for
general \(j\) are \(2j-2,2j-4,\ldots,2\).

We shall use the following elementary criterion.  If \(A\) is the set of
minimal roots of an upper ideal \(I\), then \(I\) is locked if and only if
every root which covers an element of \(A\) lies above at least two
elements of \(A\).  Indeed, after deleting \(\alpha\in A\), a new minimal
root can only be a cover of \(\alpha\); and it is new minimal precisely
when \(\alpha\) was the unique member of \(A\) below it.  In simply-laced
type a cover and the root it covers have inner product one.

\begin{lemma}[canonical fork peeling]\label{lem:d-canonical-peel}
The height tail \(T_{2j}\) is carried to \(K_{n,j}\) by Yu moves.
The ideal \(K_{n,j}\) is locked and balanced.
\end{lemma}

\begin{proof}
For \(0\leq a\leq j-2\), set
\[
 C_a=2n-2j-2a,
 \qquad n-2j+1\leq p\leq n-j-a-1.                          \tag{5.3}
\]
Delete, in increasing order of \(a\), the pairs
\[
                 y_{p,C_a-p},\qquad y_{p,C_a-p-1},          \tag{5.4}
\]
and, for fixed \(a\), in decreasing order of \(p\).  The two roots in
(5.4) have inner product one.  The cover relations
\[
 y_{pq}\lessdot y_{p-1,q},\quad
 y_{pq}\lessdot y_{p,q-1},                                 \tag{5.5}
\]
away from the fork columns, together with the two usual fork covers,
show successively that the first root in (5.4) is minimal and that, after
its deletion, the second is minimal.  The pairs in (5.4) are disjoint,
and their union is exactly (5.1).  This proves the first assertion.

The minimal antichain of the resulting ideal is
\[
 A_{n,j}=
 \{x_{p,p+2j}:1\leq p\leq n-2j\}
 \mathbin{\cup}\{y_{n-2j,n}\}.                              \tag{5.6}
\]
This follows either from (5.1), or directly from the cover relations.
Each upper cover of an interior member \(x_{p,p+2j}\) lies above the
neighbouring member of the displayed \(x\)-string.  At its right end the
two fork covers lie above both \(x_{n-2j,n}\) and \(y_{n-2j,n}\); the left
end is checked in the same way using \(x_{1,1+2j}\) and its neighbour.
Thus every cover of a member of \(A_{n,j}\) lies above a second member of
\(A_{n,j}\).  The criterion above proves that \(K_{n,j}\) is locked.
Finally each deleted pair has opposite height parity.  To evaluate the
height tail, use the standard height--exponent identity
\[
 \sum_{\beta\in R^+}q^{\operatorname{ht}(\beta)}
   =\sum_{i=1}^{\operatorname{rank}R}(q+q^2+\cdots+q^{m_i}),
\]
where \(m_i\) are the exponents of \(R\).  Every exponent of
\(D_{2m}\) is odd, so each nonempty alternating tail beginning at the
even height \(2j\) has sum zero.  Consequently
\[
 \sum_{\beta\in T_{2j}}(-1)^{\operatorname{ht}(\beta)+1}=0;
\]
hence \(K_{n,j}\) is balanced.
\end{proof}

\begin{lemma}[locked-boundary classification]\label{lem:d-boundary}
If \(I\) is a nonempty balanced locked upper ideal in \(D_{2m}^+\), then
\[
                              I=K_{n,j}
 \quad\text{for a unique }1\leq j<m.                       \tag{5.7}
\]
\end{lemma}

\begin{proof}
We give the boundary argument in a form which keeps track of the fork.
The roots \(x_{pq}\) and \(y_{pq}\), with the covers (5.5) and their
\(x\)-analogues, form two triangular arrays joined along the two fork
columns.  Adjacent cells have opposite colors.
More explicitly, the seam consists of the diamonds
\[
\begin{matrix}
 &y_{p,n-1}&\\[-1mm]
 x_{p,n}&&y_{p,n}\\[-1mm]
 &x_{p,n-1}&
\end{matrix}                                                  \tag{5.8}
\]
in which each lower entry is covered by the adjacent upper entries.
The intersection of an upper ideal with one such diamond is one of
\[
 \varnothing,\quad \{t\},\quad\{t,l\},\quad\{t,r\},\quad
 \{t,l,r\},\quad\{t,l,r,b\},                               \tag{5.9}
\]
where \(t,l,r,b\) denote its top, left, right and bottom cells.

Let \(A=\min(I)\).  In each triangular array, \(I\) is the region lying
weakly above \(A\); we call the lattice path separating this region from
its complement the \emph{lower boundary} of \(I\).  Thus the roots of
\(A\) are exactly the roots of \(I\) met along this boundary.

First consider a portion of the boundary away from the fork seam.  If
the boundary makes an inward turn at \(\beta\in A\), then one of the
upper covers \(\gamma=\beta+\alpha\), with \(\alpha\) simple, lies above
no other member of \(A\).  After \(\beta\) is deleted, \(\gamma\)
therefore becomes minimal.  Since \(\gamma-\beta\) is simple, the two
roots have inner product one and their successive deletion is a Yu
move.  Lockedness excludes such an inward corner.  Consequently every
ordinary portion of the boundary follows a diagonal of constant root
height.

At the fork, sweep the boundary from the two outer arms inward.  Pair
successive cells in each maximal alternating strip.  Directly from the
four fork covers one obtains the following alternatives:
\[
\begin{array}{c|c|c}
\text{local passage}&\text{paired remainder}&\text{next boundary}\\ \hline
\text{straight}&\varnothing&\text{same height diagonal}\\
\text{one inward step}&\varnothing&\text{remove one pair (5.4)}\\
\text{premature turn}&\text{one odd cell}&\text{nested boundary}\\
\text{outward corner}&\text{Yu pair}&\text{forbidden}.
\end{array}                                                  \tag{5.10}
\]
Figure~\ref{fig:d-fork-passages} is a schematic picture of these four
possibilities.  The shaded band is the fork seam and the heavy line is
the lower boundary.  The horizontal guide lines represent consecutive
height diagonals.  The picture records only the behavior of the
boundary; the cover relations themselves are those in~(5.8).

\begin{figure}[ht]
\centering
\begin{tikzpicture}[x=0.72cm,y=0.72cm,
                    boundary/.style={very thick},
                    guide/.style={thin,densely dotted}]
  \foreach \s/\name in {0/{straight},4.1/{one inward step},
                         8.2/{premature turn},12.3/{outward corner}} {
    \begin{scope}[xshift=\s cm]
      \fill[black!8] (1.25,-.85) rectangle (1.75,.85);
      \draw[guide] (0,.55)--(3,.55);
      \draw[guide] (0,0)--(3,0);
      \draw[guide] (0,-.55)--(3,-.55);
      \node[font=\scriptsize] at (1.5,1.13) {\name};
    \end{scope}
  }
  \begin{scope}
    \draw[boundary] (0,0)--(3,0);
  \end{scope}
  \begin{scope}[xshift=4.1cm]
    \draw[boundary] (0,.55)--(1.25,.55)--(1.75,-.55)--(3,-.55);
  \end{scope}
  \begin{scope}[xshift=8.2cm]
    \draw[boundary] (0,.55)--(1.5,.55)--(1.5,0)--(.35,0);
    \fill (1.5,0) circle (1.8pt);
    \node[font=\tiny,below right] at (1.5,0) {odd endpoint};
  \end{scope}
  \begin{scope}[xshift=12.3cm]
    \draw[boundary] (0,0)--(1.25,0)--(1.75,.55)--(3,.55);
    \fill (1.25,0) circle (1.8pt);
    \fill (1.75,.55) circle (1.8pt);
    \node[font=\tiny,below] at (1.5,-.08) {Yu pair};
  \end{scope}
\end{tikzpicture}
\caption{The four possible local passages of the lower boundary through
the fork seam.  The diagram is schematic: the exact fork diamonds and
their cover relations are displayed in~\emph{(5.8)}.}
\label{fig:d-fork-passages}
\end{figure}
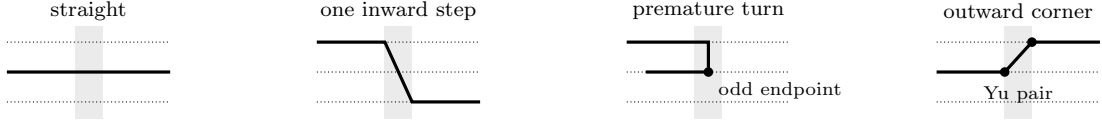

Indeed, substituting the six possibilities (5.9), the two asymmetric
possibilities are exchanged by the fork automorphism and give the second
line; \(\{t,l,r\}\) gives the straight line; and a change between two
successive diamonds gives respectively the third or fourth line.  Thus
the table exhausts the local possibilities; it is not an appeal to a
picture beyond the four stated cover relations.
Here pairing means pairing consecutive roots along each maximal portion
of the boundary between two visits to the fork seam.  Consecutive roots
have opposite height parity, so every such pair has total color zero.
In the third line, the pairing leaves one endpoint.  That endpoint has
odd height: the distance between the two relevant visits to the fork
seam is \(n-4\), which is even because \(n=2m\).  The remaining boundary
encloses a strictly smaller region, so the same argument may be applied
to it.  The unpaired odd root lies outside every later nested region and
can never enter a subsequent pair.  Thus every premature turn makes a
strictly positive contribution, while straight passages and inward
steps contribute zero.  Consequently
Consequently
\[
             \sum_{\beta\in I}(-1)^{\operatorname{ht}(\beta)+1}\geq0,
                                                                    \tag{5.11}
\]
and equality permits only the first two lines of (5.10).

Thus a balanced locked boundary starts on one even height diagonal,
say height \(2j\), and at its successive encounters with the fork takes
all possible inward steps.  The pairs removed at the \(a\)-th encounter
are precisely (5.4).  The resulting ideal is therefore \(K_{n,j}\).
If \(j\geq m\), put \(p=n-j-1\).  Then \(y_{p,p+2}\) is minimal and,
after its deletion, \(y_{p,p+1}\) is minimal; their inner product is one.
Hence lockedness forces \(j<m\).  Uniqueness follows from the least
height occurring in the ideal.
\end{proof}

\begin{proposition}\label{prop:d-locked-classification}
The nonempty balanced locked ideals of \(D_{2m}\) are precisely
\[
                         K_{n,1},K_{n,2},\ldots,K_{n,m-1}.
\]
Each is connected by Yu moves to the corresponding height tail
\(T_2,T_4,\ldots,T_{2m-2}\).
\end{proposition}

\begin{proof}
Combine Lemmas~\ref{lem:d-canonical-peel} and \ref{lem:d-boundary}.
\end{proof}

\section{Canonical height-tail representatives in type \(D_{2m}\)}
\label{sec:d-even-tails}

Use the standard positive system
\[
 R^+=\{e_p-e_q,e_p+e_q:1\leq p<q\leq 2m\}
\]
and put
\[
 I_{2j}=\{\beta\in R^+:\operatorname{ht}(\beta)\geq2j\},
 \qquad 1\leq j<m.
\]

\begin{lemma}[the \(D\)-tail count]\label{lem:d-tail-count}
At a generic point of the cell belonging to \(I_{2j}\), the negative
index of the long form on \(\operatorname{Sym}^2V\) is
\[
 n_2(m,j)=
 \begin{cases}
 j(4m-5j+3),&
                    2j\leq m,\\[2mm]
 j(4m-5j-1)+2m,&
                    m<2j\ \text{and}\ 4j\leq3m-2,\\[2mm]
 4(m-j)(m-j+1),&
                    4j>3m-2.
 \end{cases}                                                \tag{6.1}
\]
In particular \(n_2(m,j)>0\).
\end{lemma}

\begin{proof}
We give the index calculation, including the conventions needed to check
it without an intertwining matrix.  The heights are
\[
\begin{aligned}
 \operatorname{ht}(e_p-e_q)&=q-p,\\
 \operatorname{ht}(e_p+e_q)&=2(2m)-p-q
                  &&(q\leq2m-2),\\
 \operatorname{ht}(e_p+e_{2m-1})&=2m+1-p,\\
 \operatorname{ht}(e_p+e_{2m})&=2m-p.
\end{aligned}                                                \tag{6.2}
\]
Choose \(a\) with
\[
       \frac1{2j}<a<\frac1{2j-1}
\]
and evaluate at
\(\alpha_i(\nu)=a+\varepsilon^{i}\), where
\(0<\varepsilon\ll1\).  Thus membership in the cell is decided first by
height and equal heights are ordered lexicographically.

For later reference define, for \(1\leq r<s\leq N\),
\[
\begin{aligned}
 H_N^-(r,s)&=s-r,\\
 H_N^+(r,s)&=
 \begin{cases}
 2N-r-s,&s\leq N-2,\\
 N+1-r,&s=N-1,\\
 N-r,&s=N.
 \end{cases}
\end{aligned}                                                 \tag{6.3}
\]
These are precisely the heights of \(e_r-e_s\) and \(e_r+e_s\)
in \(D_N\).

For \(\beta=e_p\pm e_q\), its orthogonal subsystem is
\[
 R_\beta=A_1\mathbin{\times}D_{2m-2}.                        \tag{6.4}
\]
The \(A_1\)-root is \(e_p\mp e_q\).  The positive system on the second
factor is obtained simply by deleting positions \(p,q\) from the ordered
list of coordinates.  Put
\[
 d_{pq}(r)=r-\mathbf1_{p<r}-\mathbf1_{q<r}.
\]
Here \(\mathbf1_{\mathcal C}\) denotes the indicator of a condition
\(\mathcal C\).
Thus, for \(r<s\) disjoint from \(\{p,q\}\),
\[
\begin{aligned}
 \operatorname{ht}_{R_\beta}(e_r-e_s)
 &=H_{2m-2}^-\bigl(d_{pq}(r),d_{pq}(s)\bigr),\\
 \operatorname{ht}_{R_\beta}(e_r+e_s)
 &=H_{2m-2}^+\bigl(d_{pq}(r),d_{pq}(s)\bigr).
\end{aligned}                                                 \tag{6.5}
\]
Together with the mate \(e_p\mp e_q\), of subsystem height \(1\),
these are all the positive roots orthogonal to \(\beta\).

To make the tie convention completely explicit, write a root in simple
coordinates as \(\gamma=\sum c_i(\gamma)\alpha_i\), and order roots by the
lexicographic tuple
\[
 \left(\operatorname{ht}(\gamma),
       c_1(\gamma),c_2(\gamma),\ldots,c_{2m}(\gamma)\right).  \tag{6.6}
\]
This is the order obtained from the above \(\nu\) for sufficiently
separated successive positive infinitesimals.  Write
\(\gamma\succ\beta\) for this order.

Let
\[
 k_\beta=\sum_{\substack{\gamma\in R_\beta^+\\
                          \gamma(\nu)>\beta(\nu)}}
 (-1)^{\operatorname{ht}_{R_\beta}(\gamma)+1}.
\]
Substitution of (6.3)--(6.6) gives the fully explicit sum
\[
\begin{aligned}
k_{pq}^{\sigma}
={}&\mathbf1_{\,e_p-\sigma e_q\succ e_p+\sigma e_q}\\
&+\sum_{\substack{r<s,\ \{r,s\}\cap\{p,q\}=\varnothing\\
                   \tau\in\{-,+\}\\
                   e_r+\tau e_s\succ e_p+\sigma e_q}}
(-1)^{H_{2m-2}^{\tau}(d_{pq}(r),d_{pq}(s))+1},
\end{aligned}                                                 \tag{6.7}
\]
where \(\sigma,\tau\in\{-,+\}\), with the evident interpretation of the
signs.  Hence the sum needed below is
\[
 \sum_{\sigma\in\{-,+\}}\ \sum_{p<q}
 \mathbf1_{H_{2m}^{\sigma}(p,q)\geq2j}\,
 (-1)^{H_{2m}^{\sigma}(p,q)+1}k_{pq}^{\sigma}.               \tag{6.8}
\]

We evaluate (6.8).  Partition the indices \(r,s\) according to the five
intervals cut out by \(p,q,2m-1,2m\).  On each part,
\(d_{pq}\) is affine and both height functions in (6.7) are affine.
Therefore the inner sums are finite alternating intervals.  Repeated use
of
\[
 \sum_{a\leq r\leq b}(-1)^r
 =\frac{(-1)^a+(-1)^b}{2}                                   \tag{6.9}
\]
(with value \(0\) for an empty interval), followed by the \(p,q\)-sum.
Here is an endpoint ledger for that calculation.  Put
\(u_+=\max(u,0)\) and
\(\delta=\mathbf1_{\,4j>3m-2}\).  The ordinary endpoints, the first fork
endpoint, and the second fork endpoint contribute respectively
\[
\begin{array}{c|c}
\text{endpoints}&
\displaystyle\sum_{\beta\in I_{2j}}
(-1)^{\operatorname{ht}(\beta)+1}k_\beta\ \text{ contribution}\\ \hline
\text{ordinary}&-\frac12j(4m-5j+3)\\[1mm]
\text{first fork}&(2j-m)_+\\[1mm]
\text{second fork}&-\frac12\delta(3j-2m)(3j-2m-1).
\end{array}                                                   \tag{6.10}
\]
To verify the ledger, in each of the five index intervals replace the
inner sum by (6.9).  The two displayed endpoint terms survive unless the
deleted positions \(p,q\) cross, respectively, the first or the second
fork column.  Counting the admissible \(p\)'s gives \(2j-m\) in the first
case.  In the second case the admissible pairs form a triangle of side
\(3j-2m-1\), giving
\(\binom{3j-2m}{2}\); its two orientations account for the displayed
coefficient.  All remaining endpoints give the ordinary term.  Thus,
without any parity cases left over,
\[
\begin{aligned}
 -2\sum_{\beta\in I_{2j}}
 (-1)^{\operatorname{ht}(\beta)+1}k_\beta
={}&j(4m-5j+3)-2(2j-m)_+\\
 &+\delta(3j-2m)(3j-2m-1).
\end{aligned}                                               \tag{6.11}
\]
The only changes in an endpoint occur at
\[
             2j=m,\qquad 4j=3m-2.
\]
The resulting sums are
\[
\begin{array}{c|c}
\text{range}&
\displaystyle
\sum_{\beta\in I_{2j}}
(-1)^{\operatorname{ht}(\beta)+1}k_\beta\\ \hline
2j\leq m&
 -\frac12j(4m-5j+3)\\[1mm]
m<2j,\ 4j\leq3m-2&
 -\frac12\{j(4m-5j-1)+2m\}\\[1mm]
4j>3m-2&
 -2(m-j)(m-j+1).
\end{array}                                                   \tag{6.12}
\]
Equations (6.7)--(6.11) are also a direct recipe for checking every entry
of (6.12); no
representation-theoretic classification enters.

All exponents of \(D_{2m}\) are odd.  Hence
\[
 \sum_{\beta\in I_{2j}}
 (-1)^{\operatorname{ht}(\beta)+1}=0.                        \tag{6.13}
\]
Applying the one-wall recursion and using (6.11) gives
\[
\begin{aligned}
n_2(m,j)
 &=\sum_{\beta\in I_{2j}}
 (-1)^{\operatorname{ht}(\beta)+1}
 \bigl(2m-1-2k_\beta\bigr)\\
 &=-2\sum_{\beta\in I_{2j}}
 (-1)^{\operatorname{ht}(\beta)+1}k_\beta.
\end{aligned}
\]
Now (6.12) is exactly (6.1).

Positivity is immediate in the first and third ranges.  In the middle
range write \(k=m-j>0\).  The upper inequality is \(m\leq4k-2\), and
\[
 j(4m-5j-1)+2m
 =(m-k)(5k-m-1)+2m>0;
\]
the finitely sharp boundary \(k=1\) belongs to the third range.  This
finishes the proof.
\end{proof}

\begin{corollary}\label{cor:d-tails-indefinite}
Every nonempty balanced locked endpoint \(K_{2m,j}\) is detected by
\(\operatorname{Sym}^2V\).
\end{corollary}

\begin{proof}
By Lemma~\ref{lem:d-canonical-peel}, the height-tail cell \(I_{2j}\) and
the locked cell \(K_{2m,j}\) are joined by Yu moves.  Yu's two-wall
theorem says that positivity on any fixed \(W\)-type is equivalent at the
two ends of such a move.  The height-tail form is not positive by
Lemma~\ref{lem:d-tail-count}; hence neither is the form at the locked
endpoint.  No assertion that the complete signature is preserved is needed.
\end{proof}

\section{Inheritance by Levi subgroups}
\label{sec:downward-inheritance}

Say that a root system \(R\) has property \((\mathcal P)\) if every regular
Hermitian spherical parameter whose normalized forms on \(V_R\) and
\(\operatorname{Sym}^2V_R\) are positive definite is unitary.

\begin{lemma}[downward inheritance]\label{lem:downward-inheritance}
Assume that the longest element of \(R\) acts by \(-1\).  If \(R\) has
property \((\mathcal P)\), then every standard Levi root system \(R_J\)
has property \((\mathcal P)\).
\end{lemma}

\begin{proof}
We use normalized induction in stages and Frobenius reciprocity for
graded affine Hecke algebras in their compact pictures; these are the
standard constructions underlying the graded reduction in
\cite{Lusztig}.
Let \(\lambda\) be a regular Hermitian spherical parameter for \(R_J\), and
suppose its forms on \(V_J\) and \(\operatorname{Sym}^2V_J\) are positive.
All assertions are constant on the Levi cell, so we may move \(\lambda\)
inside that cell away from the finitely many additional equations
\(\beta(\lambda)=\pm1\), \(\beta\in R\setminus R_J\).  Use the trivial
(hence unitary) character on \(V_J^\perp\), and form the parabolically
induced endpoint
\[
 X_R(\nu_0)=\operatorname{Ind}_{\mathbb H_J}^{\mathbb H_R}
       (X_J(\lambda)\boxtimes\mathbf1),                      \tag{7.1}
\]
where \(\nu_0=\lambda+0\in V_J\oplus V_J^\perp\).  The choice of
\(\lambda\) makes (7.1) irreducible.

At the endpoint, compact-picture Frobenius reciprocity identifies each ambient
multiplicity form with the orthogonal sum of the Levi multiplicity forms
in the restriction of the ambient \(W\)-module.  Since \(W_J\) fixes
\(V_J^\perp\) pointwise,
\[
\begin{aligned}
 V_R|_{W_J}&=V_J\oplus V_J^\perp,\\
 \operatorname{Sym}^2V_R|_{W_J}
 &=\operatorname{Sym}^2V_J
   \oplus(V_J\otimes V_J^\perp)
   \oplus\operatorname{Sym}^2V_J^\perp.
\end{aligned}                                               \tag{7.2}
\]
The middle summand is a direct sum of copies of \(V_J\), and the last is
trivial.  All summands in (7.2) therefore have positive multiplicity form:
this is the hypothesis for the first two kinds, and the spherical
normalization for the trivial kind.  Hence both ambient test forms at
\(\nu_0\) are positive definite.

Move \(\nu_0\) slightly in a real direction transverse to the parabolic
face and into a Weyl chamber.  The ambient systems to which this lemma is
applied have \(w_0=-1\), so every real displacement is Hermitian.
Positivity is open, so the two ambient test forms remain positive; choose
the displacement off every reducibility hyperplane.  By
property \((\mathcal P)\) for \(R\), the resulting irreducible ambient
spherical module is unitary.  Letting the displacement tend to zero shows
that (7.1) has a positive semidefinite invariant form.  It is nondegenerate
because its radical is an invariant submodule, while the normalization on
the spherical vector is nonzero; irreducibility therefore makes the
radical zero.  Hence the limiting form is positive definite.

Finally, at the exact parabolic endpoint (7.1), the identity-coset
subspace in the normalized compact picture has invariant form equal, up
to a positive scalar, to the form on
\(X_J(\lambda)\boxtimes\mathbf1\).  Positivity of (7.1) therefore implies
positivity of the form on \(X_J(\lambda)\).  Thus the moved parameter is
unitary.  Finally, the normalized Levi form is nonsingular throughout the
original regular cell, so its signature is constant there.  Moving back
inside that cell proves that the original \(\lambda\) is unitary, and
hence proves property \((\mathcal P)\) for \(R_J\).
\end{proof}

\begin{corollary}\label{cor:ADE-downward}
Once property \((\mathcal P)\) is proved for every \(D_{2m}\) and for
\(E_8\), it follows for every simply-laced root system.  Indeed,
\(D_{2m-1}\) is a standard Levi of \(D_{2m}\), every \(A_r\) is a standard
Levi of a sufficiently large even \(D\), and \(E_6,E_7\) are standard
Levi root systems of \(E_8\).
\end{corollary}

\section{The analytic assertion in type \(E_8\)}
\label{sec:e8-analytic}

We use Bourbaki's numbering of the simple roots of \(E_8\):
\[
\begin{array}{ccccccccccccc}
 &&&&\alpha_2\\[-1mm]
 &&&&|\\[-1mm]
\alpha_1&-&\alpha_3&-&\alpha_4&-&\alpha_5&-&
\alpha_6&-&\alpha_7&-&\alpha_8 .
\end{array}
\]
Thus \(\alpha_4\) is the trivalent vertex, and the three arms issuing
from it are
\[
 \{\alpha_2\},\qquad
 \{\alpha_3,\alpha_1\},\qquad
 \{\alpha_5,\alpha_6,\alpha_7,\alpha_8\}.
\]
Every simple-root coordinate vector in this section is written in the
standard Bourbaki order
\[
 (\alpha_1,\alpha_2,\alpha_3,\alpha_4,
   \alpha_5,\alpha_6,\alpha_7,\alpha_8).
\]
For a positive root \(\beta\), put
\[
 \epsilon_\beta=(-1)^{\operatorname{ht}(\beta)+1},\qquad
 R_\beta=R\cap\beta^\perp.
\]
At a generic \(\nu\), define
\[
 k_\beta(\nu)=
 \sum_{\substack{\gamma\in R_\beta^+\\
                  \gamma(\nu)>\beta(\nu)}}
 (-1)^{\operatorname{ht}_{R_\beta}(\gamma)+1}.              \tag{8.1}
\]

\begin{proposition}\label{prop:e8-locked-indefinite}
On each of the four nonempty reflection-balanced locked \(E_8\) cells, the
normalized long-intertwining form on \(\operatorname{Sym}^2V\) is
indefinite.  Its negative index is, in the order
\[
 (|I|,\min\operatorname{ht}I)
 =(96,4),(64,8),(112,2),(32,14),
\]
respectively
\[
                         18,\quad16,\quad14,\quad6.          \tag{8.2}
\]
The four ideals are generated by the following minimal antichains, in
the Bourbaki simple-root coordinate order fixed above:
\[
\begin{array}{c|l}
96&(0,0,0,0,1,1,1,1),(0,0,0,1,1,1,1,0),
    (0,0,1,1,1,1,0,0),\\
  & (1,0,1,1,1,0,0,0),(0,1,0,1,1,1,0,0),
    (1,1,1,1,0,0,0,0)\\
64&(1,1,1,1,1,1,1,1),(1,1,1,2,1,1,1,0),
    (1,1,1,2,2,1,0,0),(1,1,2,2,1,1,0,0)\\
112&(0,0,0,0,0,0,1,1),(0,0,0,0,0,1,1,0),
     (0,0,0,0,1,1,0,0),(1,0,1,0,0,0,0,0),\\
   &(0,0,0,1,1,0,0,0),(0,0,1,1,0,0,0,0),
     (0,1,0,1,0,0,0,0)\\
32&(1,1,2,3,2,2,2,1),(1,1,2,3,3,2,1,1),
    (1,2,2,3,2,2,1,1).
\end{array}                                                \tag{8.2a}
\]
\end{proposition}

\begin{proof}
The one-wall recursion gives
\[
 n_2(I)=\sum_{\beta\in I}
 \epsilon_\beta\bigl(7-2k_\beta(\nu)\bigr).                 \tag{8.3}
\]
Because the reflection form is positive on the four cells,
\[
 \sum_{\beta\in I}\epsilon_\beta=0.
\]
If \(K_{\rm ev}\) and \(K_{\rm odd}\) denote the sums of \(k_\beta\)
over roots of even and odd height, respectively, (8.3) therefore reduces
to
\[
                     n_2(I)=2(K_{\rm ev}-K_{\rm odd}).       \tag{8.4}
\]

The following table performs the remaining integer arithmetic.  The
parameter column lists the eight simple-root values of \(1000\nu\).
For each row the defining inequalities of the indicated cell are strict,
and the values of orthogonal roots are pairwise distinct.
\[
\begin{array}{c|c|c|c|c|c}
(|I|,\min\operatorname{ht}I)&1000(\alpha_i(\nu))_{i=1}^8&
N_{\rm ev}=N_{\rm odd}&K_{\rm ev}&K_{\rm odd}&n_2\\ \hline
(96,4)&(538,182,181,177,177,547,183,179)&48&101&92&18\\
(64,8)&(436,88,87,83,84,91,89,90)&32&62&54&16\\
(112,2)&(661,675,661,698,657,637,665,685)&56&132&125&14\\
(32,14)&(30,67,32,69,64,65,65,270)&16&14&11&6
\end{array}                                                   \tag{8.5}
\]
For completeness, every entry in (8.5) is obtained as follows.  Generate
the 120 positive roots from the Bourbaki Cartan matrix determined by the
displayed diagram.  For each
\(\beta\), the 63 positive roots orthogonal to \(\beta\) form \(E_7^+\);
its seven indecomposable roots give the height in (8.1).  Compare their
rational values at the parameter in (8.5), take the alternating sum, and
then sum separately over the two parities of
\(\operatorname{ht}(\beta)\).  All operations are integer comparisons and
integer additions.  They give the two \(K\)-columns in (8.5), and (8.4)
gives the last column.  In particular every last entry is positive, proving
indefiniteness.
\end{proof}

\begin{corollary}\label{cor:e8-converse}
If an \(E_8\) cell is positive on both \(V\) and
\(\operatorname{Sym}^2V\), it is one of Yu's sixteen cells.
\end{corollary}

\begin{proof}
Reflection positivity says that its root ideal is balanced.  Exhaustive
root-poset inspection shows that every nonempty balanced locked ideal is
one of the four ideals of
Proposition~\ref{prop:e8-locked-indefinite}.  Yu's two-wall theorem makes
positivity on each fixed \(W\)-type equivalent at the two ends of a move.
Starting from any balanced ideal, perform moves until a locked endpoint is
reached.  The endpoint cannot be nonempty if the symmetric-square form is
positive.  Hence it is empty, and the original cell is a Yu cell.

For reference, the finite certificate is as follows.  Antichain
enumeration gives 205 balanced ideals and exactly the four nonempty locked
ideals (8.2a).  Recursive deletion gives endpoint distribution
\[
 \begin{array}{c|ccccc}
 \text{endpoint}&\varnothing&96&64&112&32\\ \hline
 \text{number of balanced ideals}&16&16&81&1&91.
 \end{array}                                               \tag{8.6}
\]
The entries sum to 205.  Testing minimality and the cover criterion for
each antichain in (8.2a), followed by the two possible deletions at every
unlocked node, is an integer root-poset certificate for the assertion.
\end{proof}

\begin{remark}
Proposition~\ref{prop:e8-locked-indefinite} proves the analytic conclusion
needed for \(E_8\) directly.  It does not assert the stronger and
noncanonical statement that the form of a transported \(D_4\) subsystem is
literally a multigraded Jantzen quotient of the \(E_8\) form.  The latter
requires specifying a degeneration and proving an additional localization
identity; it is unnecessary for the \(E_8\) converse.
\end{remark}

\section{The simply-laced converse}
\label{sec:simply-laced-converse}

\begin{theorem}\label{thm:simply-laced-converse}
Let \(R\) be an irreducible simply-laced root system and let \(\nu\) be a
generic Hermitian spherical parameter.  Suppose that the normalized
Hermitian forms on the reflection representation \(V\) and on every
irreducible constituent of \(\operatorname{Sym}^2V\) are positive definite.
Then \(\nu\) lies in one of Yu's regions and the spherical representation
is unitary.
\end{theorem}

\begin{proof}
First take \(R=D_{2m}\).  Positivity on \(V\), together with the reflection
signature formula, says that the upper ideal \(I(\nu)\) is balanced.
Apply Yu moves until none remains.  Yu's two-wall theorem preserves
positivity (in both directions) on every fixed \(W\)-type.  If the endpoint is nonempty,
Proposition~\ref{prop:d-locked-classification} identifies it with one of
the ideals \(K_{2m,j}\).  Corollary~\ref{cor:d-tails-indefinite} says that
the form on \(\operatorname{Sym}^2V\) is then indefinite, contradicting the
hypothesis.  Hence the endpoint is empty.  This is exactly the condition
that the original cell is one of Yu's cells, and Yu's theorem gives
unitarity.

For \(R=E_8\), the same conclusion is
Corollary~\ref{cor:e8-converse}: the four nonempty balanced locked endpoints
have nonzero negative index on \(\operatorname{Sym}^2V\).

Property \((\mathcal P)\) therefore holds for every even \(D\) and for
\(E_8\).  Corollary~\ref{cor:ADE-downward} passes it to every odd \(D\),
every \(A\), and to \(E_6,E_7\).  These are all irreducible simply-laced
root systems.
\end{proof}

\begin{remark}
The proof uses no catalogue of relevant \(W\)-types.  Its infinite-family
input is the uniform wall recursion and fork-boundary calculation in even
type \(D\); the exceptional input is the direct finite \(E_8\) wall
calculation.  The theorem for \(E_6\) and \(E_7\) is inherited from \(E_8\),
not proved by a separate enumeration.
\end{remark}

\section{The multiple-bond wall problem in types \(B_n\) and \(C_n\)}
\label{sec:bc-wall}

Write the positive roots as
\[
 e_p-e_q,\quad e_p+e_q\quad(p<q),\qquad
 z_p=
 \begin{cases}e_p,&B_n,\\2e_p,&C_n.\end{cases}              \tag{10.1}
\]
The two systems have the same Weyl group, but different affine wall
coordinates and different root posets.

\subsection{Why the simply-laced recursion does not extend}

At the wall \(\beta(\nu)=1\), the radical of the symmetric-square form is
still
\[
                 \beta\odot\beta^\perp,
                                                               \tag{10.2}
\]
of dimension \(n-1\).  Across a multiple bond, however, its first Jantzen
form is not the equal-parameter reflection form of
\(R\cap\beta^\perp\).  Nonorthogonal roots of the other length contribute
to the quotient.  For example, at a middle coordinate-root wall in
\(B_3\), the exact symmetric-square jump is \(2\), whereas the signature
difference of the equal-parameter \(B_2\) reflection form is \(0\).
Consequently formula~\ref{cor:sym2-signature} must not be used unchanged
in types \(B\) and \(C\).

\subsection{The two symmetric-square constituents}

For the hyperoctahedral Weyl group,
\[
 \operatorname{Sym}^2V
 =\mathbf1\oplus U_n\oplus X_n,                              \tag{10.3}
\]
where \(U_n\) is the traceless diagonal subspace spanned by the
\(e_i^2\), and \(X_n\) is the off-diagonal subspace spanned by
\(e_i\odot e_j\), \(i<j\).  Both subspaces are \(W\)-stable.  Hence the
criterion asks precisely for positivity of the forms on \(U_n\) and
\(X_n\), in addition to the reflection form; the trivial summand is
positive by normalization.

This decomposition also locates the multiple-bond correction.  A
coordinate reflection acts trivially on \(U_n\), while the radical at a
coordinate-root wall lies entirely in \(X_n\):
\[
 z_p\odot z_p^\perp
   =\operatorname{span}\{e_p\odot e_q:q\ne p\}\subset X_n.  \tag{10.4}
\]
Thus ordinary \(A_2\) wall pairs are governed by the simply-laced
calculation on the roots \(e_p\pm e_q\); the new analytic input is the
signed-permutation Jantzen form on the space in (10.4).

\begin{remark}[how the multiple bond is handled]
A direct coordinate-wall Jantzen calculation would have to combine the
factors from each \(B_2\)-string
\[
 e_p-e_q,\quad z_p,\quad e_p+e_q
 \qquad(q\ne p).                                             \tag{10.5}
\]
Lemma~\ref{lem:b2-derivatives} supplies this local calculation.  For the
global converse, Section~\ref{sec:bc-diagonal-strategy} gives a sharper
argument on the traceless-diagonal constituent alone, using a cofactor
and an exact palindromic outer-chain matrix coefficient.
\end{remark}

\section{The \(B_2\) Jantzen table}
\label{sec:b2-jantzen}

The correction at the multiple bond is completely visible in rank two.
Write \(x>y>0\) for the coordinate values.  Use the polynomial basis
\(e_1^2,e_1e_2,e_2^2\) of \(\operatorname{Sym}^2V\).

\begin{lemma}[rank-two derivatives]\label{lem:b2-derivatives}
Up to multiplication by a positive function, the first Jantzen forms are
as follows:
\[
\begin{array}{c|c|c}
\text{wall}&\text{radical line}&\text{first derivative}\\ \hline
x=1\quad(B_2)&\mathbb R(e_1e_2)&\dfrac{y-1}{y+1}\\[2mm]
2x=1\quad(C_2)&\mathbb R(e_1e_2)&\dfrac{2y-1}{2y+1}\\[2mm]
x-y=1&\mathbb R(e_1^2-e_2^2)&\dfrac{y}{y+1}\\[2mm]
x+y=1&\mathbb R(e_1^2-e_2^2)&\dfrac{y}{y-1}.
\end{array}                                                  \tag{11.1}
\]
The formulas with the two coordinates interchanged are obtained by
symmetry.
\end{lemma}

\begin{proof}
For \(B_2\), take the reduced word \(s_1s_2s_1s_2\).  Its root sequence is
\[
 e_1-e_2,\quad e_1,\quad e_1+e_2,\quad e_2.
\]
For \(C_2\), replace the middle coordinate roots by \(2e_1,2e_2\).
Insert the corresponding three-by-three reflection matrices in
\[
 A(\nu)=\prod_{\beta}
       \frac{1+\beta(\nu)s_\beta}{1+\beta(\nu)}.
\]
On each displayed wall the kernel is the line in the middle column of
(11.1).  Differentiation in the coordinate
\(\beta(\nu)-1\), followed by restriction to that line, gives (11.1).
All omitted factors are squares of nonzero rational functions and hence
positive.
\end{proof}

\begin{corollary}\label{cor:b2-wall-signs}
At a coordinate-root wall, the sign on the \(e_p\odot e_q\) direction
changes precisely when the other coordinate-root wall has already been
crossed.  At the two long-root walls, the orientations are opposite.
\end{corollary}

\begin{remark}
This table explains the failure of the naïve orthogonal-subsystem formula.
The first two rows depend on a nonorthogonal root of the other length.
In higher rank, the coordinate-wall quotient on
\(\operatorname{span}\{e_p\odot e_q:q\ne p\}\) therefore begins with a
diagonal form whose \(q\)-th sign records whether \(z_q(\nu)>z_p(\nu)\);
ordinary \(A_2\) and commuting-square intersections then transport this
form along the wall.
\end{remark}

\section{Determinant obstruction and rank descent in \(B/C\)}
\label{sec:bc-determinant}

For an upper ideal \(I\), let
\[
 L(I)=\#\{\beta\in I:\beta=e_p\pm e_q\},\qquad
 S(I)=\#\{\beta\in I:\beta=z_p\}.                            \tag{12.1}
\]

\begin{lemma}[the two determinant characters]\label{lem:bc-determinants}
On the cell belonging to \(I\),
\[
\begin{aligned}
 \operatorname{sgn}\det A_{U_n}
   &=(-1)^{L(I)},\\
 \operatorname{sgn}\det A_{X_n}
   &=(-1)^{(n-2)L(I)+(n-1)S(I)}.
\end{aligned}                                               \tag{12.2}
\]
Consequently the symmetric-square criterion fails whenever either exponent
in (12.2) is odd.
\end{lemma}

\begin{proof}
For any \(W\)-module \(E\),
\[
 \det A_E(\nu)=
 \prod_{\beta\in R^+}
 \left(\frac{1-\beta(\nu)}{1+\beta(\nu)}\right)^{
       \dim E^{-s_\beta}}.                                  \tag{12.3}
\]
A coordinate reflection acts trivially on \(U_n\), while a long-root
reflection acts as a transposition and has a one-dimensional minus space.
On \(X_n\), a coordinate reflection has minus multiplicity \(n-1\), and a
long-root reflection has minus multiplicity \(n-2\).  Taking signs in
(12.3) gives (12.2).
\end{proof}

\begin{lemma}[shifted-boundary parity descent]\label{lem:bc-parity-descent}
Let \(I\ne\varnothing\) be a reflection-balanced locked upper ideal in
type \(B_n\) or \(C_n\).  If both exponents in (12.2) are even, then the
shifted boundary has a removable outer rank strip.  After deleting that
strip and relabelling, one obtains a nonempty reflection-balanced locked
upper ideal \(I'\) of the same root-system convention in rank \(n-1\).
In addition,
\[
                              L(I')\equiv L(I)+1\pmod2.      \tag{12.4}
\]
\end{lemma}

\begin{proof}
Put \(a_{pq}=e_p-e_q\), \(b_{pq}=e_p+e_q\).  In simple-root
coordinates the three kinds of roots are
\[
\begin{array}{c|c|c}
&B_n&C_n\\ \hline
a_{pq}&\alpha_p+\cdots+\alpha_{q-1}
       &\alpha_p+\cdots+\alpha_{q-1}\\
b_{pq}&\alpha_p+\cdots+\alpha_{q-1}
       +2\alpha_q+\cdots+2\alpha_n
       &\alpha_p+\cdots+\alpha_{q-1}
       +2\alpha_q+\cdots+2\alpha_{n-1}+\alpha_n\\
z_p&\alpha_p+\cdots+\alpha_n
       &2\alpha_p+\cdots+2\alpha_{n-1}+\alpha_n.
\end{array}                                                \tag{12.5}
\]
Thus every cover used below can be checked by adding a single simple-root
coordinate and asking whether the resulting vector is one of the vectors
in (12.5).

We call the roots \(a_{pq}\) and \(b_{pq}\) \emph{noncoordinate roots};
they are exactly the roots counted by \(L(I)\).  This terminology works
in both types: these roots are long in \(B_n\) and short in \(C_n\).
The roots \(z_p\), counted by \(S(I)\), will be called coordinate roots.

Let \(R_{n-1}\) be the standard subsystem supported on
\(e_2,\ldots,e_n\).  The roots outside \(R_{n-1}\) are
\[
 a_{1q},\quad b_{1q}\quad(2\leq q\leq n),
 \qquad z_1;
\]
they form the \emph{outer rank strip}.  Draw the roots \(a_{pq}\) in
the ordinary staircase and the roots \(b_{pq},z_p\) in the shifted
staircase, with the two staircases joined along their terminal diagonal.
Adjacent noncoordinate cells have opposite reflection colors.

Let \(A=\min(I)\).  Away from the terminal diagonal, the lower boundary
of \(I\) cannot make an inward corner.  Indeed, at such a corner some
\(\beta\in A\) has an upper cover \(\gamma=\beta+\alpha\), with
\(\alpha\) simple, which lies above no other member of \(A\).  Deleting
\(\beta\) makes \(\gamma\) minimal, so their successive deletion is a
Yu move.  This contradicts lockedness.  Hence every ordinary portion of
the boundary follows a height diagonal.

Trace the outermost boundary strip toward the shifted diagonal and pair
consecutive noncoordinate cells.  Each pair has opposite colors.  At
the end of a shifted row there are two possible cover directions: a
continuation in that row and a passage to the next row, the latter ending
in a coordinate root at the terminal diagonal.  According to which of
these two continuations is present, there are four states.  In the table,
\(\ell\) denotes one unpaired noncoordinate cell and \(c\) one unpaired
coordinate cell; these are parity-bookkeeping symbols, not root labels.
\[
\begin{array}{c|c|c}
\text{terminal state}&\text{unpaired cells}&\text{consequence}\\ \hline
A&\ell&L(I)\equiv1\pmod2,\\
B&c+(n-2)\ell
 &(n-2)L(I)+(n-1)S(I)\equiv1\pmod2,\\
C&\varnothing&\text{rank-}(n-1)\text{ shifted boundary},\\
D&\text{an ordinary two-cell corner}&\text{a Yu pair}.
\end{array}                                                  \tag{12.6}
\]
Figure~\ref{fig:bc-terminal-states} displays the four possibilities.
The bottom black vertex is the last occupied cell in the current row.
The two upper vertices are its two possible continuations.  A filled
vertex is present in the ideal and an open vertex is absent.  The diagram
is schematic; the identities of the roots depend on the row, and their
exact cover relations are obtained from~(12.5).

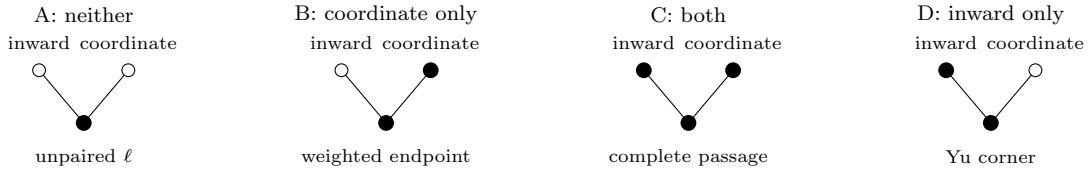
\begin{figure}[ht]
\centering
\begin{tikzpicture}[x=0.82cm,y=0.82cm,
  occupied/.style={circle,fill=black,inner sep=2.1pt},
  absent/.style={circle,draw=black,fill=white,inner sep=1.7pt},
  cover/.style={thin}]
  \foreach \s/\title in {0/{A: neither},4/{B: coordinate only},
                          8/{C: both},12/{D: inward only}} {
    \begin{scope}[xshift=\s cm]
      \node[font=\scriptsize] at (0,1.75) {\title};
      \draw[cover] (0,0)--(-.72,.85);
      \draw[cover] (0,0)--(.72,.85);
      \node[occupied] at (0,0) {};
      \node[font=\tiny,above] at (-.72,1.02) {inward};
      \node[font=\tiny,above] at (.72,1.02) {coordinate};
    \end{scope}
  }
  \begin{scope}
    \node[absent] at (-.72,.85) {};
    \node[absent] at (.72,.85) {};
    \node[font=\tiny] at (0,-.55) {unpaired \(\ell\)};
  \end{scope}
  \begin{scope}[xshift=4cm]
    \node[absent] at (-.72,.85) {};
    \node[occupied] at (.72,.85) {};
    \node[font=\tiny] at (0,-.55) {weighted endpoint};
  \end{scope}
  \begin{scope}[xshift=8cm]
    \node[occupied] at (-.72,.85) {};
    \node[occupied] at (.72,.85) {};
    \node[font=\tiny] at (0,-.55) {complete passage};
  \end{scope}
  \begin{scope}[xshift=12cm]
    \node[occupied] at (-.72,.85) {};
    \node[absent] at (.72,.85) {};
    \node[font=\tiny] at (0,-.55) {Yu corner};
  \end{scope}
\end{tikzpicture}
\caption{The four terminal states in the shifted-boundary scan.  Filled
vertices occur in the ideal and open vertices do not.}
\label{fig:bc-terminal-states}
\end{figure}

To verify completeness, begin with the last occupied coefficient vector
in a shifted row.  There are exactly two possible additions of a simple
root in (12.5): continuation in that row and passage to the next
row (at the end of the row the latter is the coordinate vector).  If both
are absent one has state A; if only the coordinate continuation is present
one has B; if both are present they form the complete passage C; and if
only the inward continuation is present its lower cell is uniquely
covered and gives D.  In state B, the unpaired part consists of one
coordinate cell and \(n-2\) noncoordinate cells.  Its contribution to
the second exponent in~(12.2) is
\[
 (n-1)+(n-2)^2\equiv(n-1)+(n-2)\equiv1\pmod2,
\]
which explains the second row of~(12.6).  Removing two consecutive
noncoordinate cells returns the scan to the same four states and changes
neither displayed parity.  This proves the table by induction along the
row; the figure is only a guide to the four possible terminal choices.

If the two parities in (12.2) vanish, only state C in (12.6) is
possible (state D is excluded by lockedness).  In state C every cell of
the outer strip is paired with a cell of the opposite color.  Removing
this paired strip leaves an upper ideal \(I'\).
More explicitly, the state-C scan and upper-ideal closure show that the
roots outside \(R_{n-1}\) form a threshold word
\[
 \{e_1+e_q:2\leq q\leq n\}\ \cup\ \{z_1\}\ \cup\
 \{e_1-e_q:r\leq q\leq n\}                                \tag{12.7}
\]
for some \(3\leq r\leq n+1\), where the last set is empty when
\(r=n+1\).  The scan in (12.6) simultaneously removes this word and the
adjacent boundary cells of \(I'\); a complete staircase diamond is
removed whenever both of its cells occur.
The same row-by-row scan gives the matching identities
\begin{equation}
\begin{aligned}
 a_{1q}\in I&\Longleftrightarrow a_{2q}\in I'\qquad(q>2),\\
 b_{1q}\in I&\Longleftrightarrow b_{2q}\in I'\qquad(q>2),\\
 z_1\in I&\Longleftrightarrow z_2\in I'.
\end{aligned}
\tag{12.8}
\end{equation}
The roots on the right are written with their original indices; after
replacing \(e_{i+1}\) by \(e_i\), they belong to the same root-system
convention in rank \(n-1\).  Indeed, failure of the first matching identity is state A, failure of the
second or third is state B, and simultaneous failure produces the
forbidden corner D.  Thus state C is equivalent to all three identities.
The same cover check shows that it is locked.  The removed strip is a
disjoint union of opposite-color pairs, so reflection balance passes from
\(I\) to \(I'\).  After
the opposite-color pairs in (12.6) have been removed, the two end cells of
the noncoordinate part of the threshold word have the same color.  Thus
the number of deleted noncoordinate cells is odd, which proves (12.4).

It remains to note that \(I'\ne\varnothing\).  If it were empty, the
threshold word (12.7) would be the whole ideal.  Substitution in (12.5)
then puts its terminal state in A or B, unless the threshold word itself
is empty.  The former alternatives contradict the two even parities and
the latter gives \(I=\varnothing\).  Both are excluded by hypothesis.
\end{proof}

\begin{remark}
Lemma~\ref{lem:bc-parity-descent} is purely combinatorial.  The analytic
completion is given in Section~\ref{sec:bc-diagonal-strategy}.  Rather
than comparing the complete forms before and after deletion, the argument
uses one cofactor of the traceless-diagonal form \(A_{U_n}\); its
maximal-parabolic matrix coefficient is evaluated by the palindromic
outer-chain calculation of Lemma~\ref{lem:bc-outer-chain}.
\end{remark}

\section{A sharper diagonal reduction in types \(B/C\)}
\label{sec:bc-diagonal-strategy}

The off-diagonal constituent is not needed at locked endpoints.  Let
\(U_n\) be the traceless diagonal constituent and let
\(S(I)\) denote the number of coordinate roots in \(I\).

Put
\[
 \mathcal D_n=\operatorname{span}\{d_i=e_i^2:1\leq i\leq n\},
 \qquad
 U_n=\left\{\sum c_i d_i:\sum c_i=0\right\},
\]
and embed
\[
 E_{n-1}=\left\{\sum_{i=2}^n c_i d_i:\sum_{i=2}^n c_i=0\right\}
 \cong U_{n-1}
\]
as a codimension-one subspace of \(U_n\).  Its Euclidean orthogonal
complement in \(U_n\) is spanned by
\[
                  w_n=(n-1)d_1-d_2-\cdots-d_n.             \tag{13.11}
\]

\begin{lemma}[the cofactor reduction]\label{lem:bc-cofactor-reduction}
Let \(H_n=A_{U_n}(\nu)\), away from its reducibility hyperplanes, and
write \(\widehat w_n=w_n/\lVert w_n\rVert\).  Then
\[
 \det\bigl(H_n|_{E_{n-1}}\bigr)
 =\det(H_n)\,
   \bigl\langle H_n^{-1}\widehat w_n,\widehat w_n\bigr\rangle .
                                                               \tag{13.12}
\]
Consequently, at a determinant-even locked endpoint, the single
inequality
\[
        \bigl\langle H_n^{-1}\widehat w_n,\widehat w_n\bigr\rangle<0
                                                               \tag{13.13}
\]
implies that \(H_n\) is indefinite.
\end{lemma}

\begin{proof}
Choose an orthonormal basis of \(U_n\) whose first \(n-2\) vectors span
\(E_{n-1}\) and whose last vector is \(\widehat w_n\).  In this basis
write
\[
 H_n=\begin{pmatrix}B&b\\ b^{t}&c\end{pmatrix},
 \qquad B=H_n|_{E_{n-1}}.
\]
Whenever \(B\) is nonsingular, Schur complementation gives
\[
 \det H_n=\det B\,(c-b^tB^{-1}b),\qquad
 \langle H_n^{-1}\widehat w_n,\widehat w_n\rangle
       =(c-b^tB^{-1}b)^{-1}.
\]
Multiplication proves (13.12).  The identity extends across the locus
where \(B\) is singular by continuity (equivalently, it is the usual
cofactor identity).  If \(L(I)\) is even, Lemma~\ref{lem:bc-determinants}
says that \(\det H_n>0\).  Under (13.13), formula (13.12) makes the
determinant of the restricted form negative.  The restriction, and hence
\(H_n\), is therefore indefinite.
\end{proof}

Since
\(H_n(\nu)^{-1}=H_n(-\nu)\), the left side of (13.13) is a single
matrix coefficient of the same normalized intertwiner at \(-\nu\).

\begin{lemma}[the outer-chain coefficient]\label{lem:bc-outer-chain}
Write the dominant parameter in Cartesian coordinates as
\(\nu=(x_1,\ldots,x_n)\), and put
\[
 c_q=x_1^2-x_q^2,\qquad d_q=(1-x_1)^2-x_q^2
 \qquad(2\leq q\leq n).
\]
With the convention that an empty product is one, define
\[
 \mathcal S_n(\nu)=
 \sum_{j=1}^{n-1}
   \frac{\prod_{r=2}^{j}c_r}{\prod_{r=2}^{j+1}d_r}
 +\frac{\prod_{r=2}^{n}c_r}{\prod_{r=2}^{n}d_r}.          \tag{13.14}
\]
Then
\[
 \frac{\langle H_n(\nu)^{-1}w_n,w_n\rangle}
      {\langle w_n,w_n\rangle}
       =\frac{n\mathcal S_n(\nu)-1}{n-1}.                 \tag{13.15}
\]
The formula is the same in types \(B_n\) and \(C_n\).
\end{lemma}

\begin{proof}
Use \(H_n(\nu)^{-1}=H_n(-\nu)\) and the parabolic factorization
\(w_0=w^Jw_{0,J}\), where \(W_J\) is the signed-permutation group on
coordinates \(2,\ldots,n\).  The vector \(w_n\) is \(W_J\)-fixed, so the
Levi intertwiner acts as the identity on it.  The minimal relative
element \(w^J\), which changes the sign of the first coordinate, has the
palindromic reduced expression
\[
 s_1s_2\cdots s_{n-1}s_n s_{n-1}\cdots s_2s_1.            \tag{13.16}
\]
On \(\mathcal D_n\), \(s_n\) is the identity and \(s_i\) interchanges
the \(i\)-th and \((i+1)\)-st coordinates.  At \(-\nu\), the parameters
on the outward and return passages across the edge \((q-1,q)\) are
\[
              a_q=-(x_1-x_q),\qquad b_q=-(x_1+x_q).
\]

It is convenient to add the constant vector and write
\(w_n+\mathbf1=(n,0,\ldots,0)\).  Every normalized transposition factor
fixes \(\mathbf1\).  During the outward passage, let \(y_j\) be the
amount of the initial value \(n\) left at the \(j\)-th coordinate after
that coordinate has completed its outward interactions.  The elementary
two-coordinate update gives
\[
 \frac{y_j}{n}=
 \begin{cases}
 \displaystyle
 \frac{\prod_{r=2}^{j}a_r}
      {\prod_{r=2}^{j+1}(1+a_r)},&j<n,\\[3mm]
 \displaystyle
 \frac{\prod_{r=2}^{n}a_r}
      {\prod_{r=2}^{n}(1+a_r)},&j=n.
 \end{cases}                                               \tag{13.17}
\]
On the return passage, the proportion of a value at coordinate \(j\)
which reaches coordinate one is given by the same expressions with
\(a_r\) replaced by \(b_r\).  Hence the final first coordinate of
\(H_n(-\nu)(w_n+\mathbf1)\) is \(n\mathcal S_n(\nu)\), because
\[
 a_rb_r=c_r,\qquad (1+a_r)(1+b_r)=d_r.
\]
Subtracting the fixed vector \(\mathbf1\), the final first coordinate of
\(H_n(-\nu)w_n\) is \(n\mathcal S_n(\nu)-1\).

All factors preserve the sum of the coordinates.  If a vector of
coordinate sum zero has first coordinate \(z\), its scalar product with
\(w_n\) is \(nz\).  Since
\(\langle w_n,w_n\rangle=n(n-1)\), formula (13.15) follows.
\end{proof}

The cofactor inequality (13.13) is therefore equivalent to
\[
                         n\mathcal S_n(\nu)<1.             \tag{13.18}
\]
As a consistency check, (13.15) agrees with the full intertwiner matrices
through rank eight.  On a state-C strip all
\(e_1+e_q\) lie in the ideal and the \(e_1-e_q\) form a threshold, so
\[
                    d_2<d_3<\cdots<d_n
\]
are initially negative and thereafter positive.  The next lemma combines
this observation with the parity information in
Lemma~\ref{lem:bc-parity-descent} to prove (13.18).  Notice that
(13.15), unlike a wall-by-wall signature formula, is an exact identity
for every generic parameter.

\begin{lemma}[sign of the outer-chain coefficient]
\label{lem:bc-outer-chain-sign}
Let \(I\) be a nonempty reflection-balanced locked ideal for which
\(L(I)\) is even.  Then
\[
             \langle H_n(\nu)^{-1}w_n,w_n\rangle<0
\]
throughout the cell of \(I\).
\end{lemma}

\begin{proof}
Reflection balance implies that \(|I|=L(I)+S(I)\) is even.  Consequently
the two exponents in (12.2) have the same parity:
\[
 (n-2)L(I)+(n-1)S(I)\equiv L(I)\pmod2.                    \tag{13.19}
\]
Indeed, this is immediate when \(n\) is odd, while for \(n\) even it
follows from \(S(I)\equiv L(I)\pmod2\).  Thus both determinant exponents
are even, and Lemma~\ref{lem:bc-parity-descent} puts the outer boundary
in state C.

Write its threshold as in (12.7).  Thus every \(e_1+e_q\) and \(z_1\)
lies in \(I\), while
\[
 e_1-e_q\in I\quad\Longleftrightarrow\quad q\geq r
\]
for some \(3\leq r\leq n+1\).  The number of noncoordinate roots removed
with the outer strip is
\[
                  (n-1)+(n-r+1)=2n-r.
\]
It is odd by (12.4), so \(r\) is odd.  Hence \(r-2\), the number of
indices \(q\) for which \(e_1+e_q\in I\) but \(e_1-e_q\notin I\), is
odd.

For the notation of Lemma~\ref{lem:bc-outer-chain}, the inequalities
defining the cell give
\[
 d_q=(1-x_1-x_q)(1-x_1+x_q)
 \begin{cases}
 <0,&q<r,\\
 >0,&q\geq r.
 \end{cases}                                               \tag{13.20}
\]
Moreover \(z_1\in I\) gives \(2x_1-1>0\): in type \(B\) one even has
\(x_1>1\), and in type \(C\) the coordinate wall is \(2x_1=1\).

Put
\[
 T_n=\prod_{q=2}^n\frac{c_q}{d_q}.
\]
Since \(c_q=x_1^2-x_q^2>0\) and an odd number \(r-2\) of the \(d_q\)
are negative, \(T_n<0\).  Finally, the sum in (13.14) telescopes:
\[
 \mathcal S_n
 =\frac{T_n-1}{2x_1-1}+T_n
 =\frac{2x_1T_n-1}{2x_1-1}.                               \tag{13.21}
\]
Here we used \(c_q-d_q=2x_1-1\), since
\[
 \frac{\prod_{r=2}^{j}c_r}{\prod_{r=2}^{j+1}d_r}
 =\frac1{2x_1-1}
 \left(
 \prod_{r=2}^{j+1}\frac{c_r}{d_r}
 -\prod_{r=2}^{j}\frac{c_r}{d_r}
 \right).
\]
Both the numerator \(2x_1T_n-1\) and hence \(\mathcal S_n\) are
negative.  In particular \(n\mathcal S_n<1\), and (13.15) proves the
required inequality.
\end{proof}

\begin{theorem}[the sharper \(B/C\) criterion]
\label{thm:bc-diagonal-converse}
For type \(B_n\) or \(C_n\), a generic Hermitian spherical parameter is
unitary if and only if its normalized forms are positive definite on
the reflection representation \(V\) and on the single irreducible
constituent
\[
 U_n=\left\{\sum_{i=1}^n c_i e_i^2:\sum_{i=1}^n c_i=0\right\}.
\]
\end{theorem}

\begin{proof}
Only the converse requires proof.  Apply the reduction procedure of
Proposition~\ref{prop:uniform-reduction} to a reflection-positive cell,
obtaining a balanced locked ideal \(I\); positivity
on the fixed type \(U_n\) is preserved in both directions by each move.
If \(I\ne\varnothing\) and \(L(I)\) is odd,
Lemma~\ref{lem:bc-determinants} gives \(\det A_{U_n}<0\), so the
\(U_n\)-form is not positive definite.  If \(L(I)\) is even,
Lemma~\ref{lem:bc-outer-chain-sign} and the cofactor identity
(13.12) show that the restriction of \(A_{U_n}\) to \(E_{n-1}\) has
negative determinant.  Again the \(U_n\)-form is indefinite.
Thus a cell positive on both \(V\) and \(U_n\) can reduce only to the
empty ideal.  It is therefore a Yu cell and is unitary.
\end{proof}

\begin{remark}
Exact signed-permutation calculations suggest the stronger formula
\[
                         n_-(A_{U_n};I)=S(I)
\]
for every nonempty reflection-balanced locked ideal.  It holds for all
such ideals through rank seven in both conventions.  The theorem uses
only the weaker consequence \(n_-(A_{U_n};I)>0\).
\end{remark}

\section{The exceptional non-simply-laced base cases}
\label{sec:f4-g2}

The preceding rank induction does not contain \(F_4\) or \(G_2\).  They
are finite base cases, just as the direct \(E_8\) calculation is the base
case for the exceptional simply-laced systems.

\begin{proposition}[type \(G_2\)]\label{prop:g2-converse}
In type \(G_2\), every reflection-positive cell which is not a Yu cell is
indefinite on the nontrivial constituent of \(\operatorname{Sym}^2V\).
\end{proposition}

\begin{proof}
Take \(\alpha _1\) long and \(\alpha _2\) short.  The six positive roots
have simple-root coordinates
\[
 (1,0),(0,1),(1,1),(1,2),(1,3),(2,3).
\]
The eight antichains give three balanced upper ideals.  Two are Yu
ideals: one is empty and the other reduces to the empty ideal by one Yu
move.  The remaining ideal is generated by \((1,1)\)
and is locked.  In the reduced word \(121212\), insert the six rank-one
factors in \(\operatorname{Sym}^2V\).  At
\((\alpha _1(\nu),\alpha _2(\nu))=(2/3,2/3)\), symmetric elimination
gives complete signature
\[
                         (n_+,n_-)=(1,2).
\]
The positive line is the trivial constituent, so the other irreducible
constituent is negative definite.  The signature is constant on the cell.
\end{proof}

\begin{proposition}[type \(F_4\)]\label{prop:f4-converse}
In type \(F_4\), every reflection-positive cell which is not a Yu cell is
indefinite on the nine-dimensional nontrivial constituent of
\(\operatorname{Sym}^2V=\mathbf1\oplus E_9\).
\end{proposition}

\begin{proof}
Use the Bourbaki simple roots.  Root-poset deletion among the twenty-four
positive roots gives 105 antichains, of which fifteen are balanced.  Two
are Yu ideals (allowing a reduction sequence of length zero).  Every
other one reduces to one of the seven locked ideals in the
following table.  An antichain is written in simple-root coordinates;
the last column is the negative index on \(\operatorname{Sym}^2V\).
\[
\begin{array}{c|c|c}
 |I|&\min(I)&n_-\\ \hline
14&\{(0,0,1,1)\}&4\\
14&\{(1,1,0,0)\}&4\\
18&\{(1,1,0,0),(0,0,1,1)\}&4\\
16&\{(1,1,0,0),(0,1,2,1)\}&6\\
12&\{(1,1,1,1),(0,1,2,1)\}&4\\
20&\{(1,1,0,0),(0,1,1,0),(0,0,1,1)\}&6\\
4 &\{(1,2,3,2)\}&2
\end{array}                                                \tag{14.1}
\]

Here is an exact verification of the last column.  For the seven rows,
respectively, take the simple-root values \(10^{-6}\) times
\[
\begin{array}{c|rrrr}
1&125001&125000&250000&875000\\
2&857144&285715&142857&142857\\
3&857144&285715&285714&857143\\
4&857144&285715&285714&285714\\
5&222223&111112&222222&555556\\
6&666670&666668&666667&666667\\
7& 83334& 83334&166667&166667
\end{array}                                                \tag{14.2}
\]
All root inequalities defining the corresponding rows of (14.1) are
strict.  In the polynomial basis
\(x_i^2,x_ix_j\;(i<j)\), form the rational matrix
\[
 A(\nu)=\prod_{k=1}^{24}
 \frac{1+\beta_k(\nu)\,\operatorname{Sym}^2(s_{i_k})}
      {1+\beta_k(\nu)},                                   \tag{14.3}
\]
where \(i_1\cdots i_{24}\) is a reduced word for the long element and
\(\beta_k=s_{i_1}\cdots s_{i_{k-1}}\alpha_{i_k}\).
Multiplication by the invariant Gram matrix
\(\operatorname{diag}(2,2,2,2,1,\ldots,1)\) makes (14.3) symmetric.
Exact rational symmetric elimination gives the seven indices displayed
in (14.1).  Thus no rounding or numerical eigenvalue decision is involved.

The trivial summand has positive normalized form.  Every positive entry
in the last column therefore belongs to \(E_9\), proving the assertion.
Yu moves preserve positivity in both directions on each fixed \(W\)-type,
so the same nonpositivity conclusion holds for every ideal reducing to
one of these seven endpoints.

The exhaustion is certified entirely in simple-root coordinates: generate
the twenty-four positive roots, enumerate the 105 antichains, retain the
fifteen whose alternating height sum is zero, and test the cover criterion
for lockedness.  The seven nonempty survivors have precisely the minimal
antichains in (14.1); the other balanced ideals either are empty or expose
a Yu pair.  Thus the calculation involves only comparisons of four-tuples
of nonnegative integers.
\end{proof}

\begin{corollary}\label{cor:exceptional-nonsimply-converse}
The reflection and second-symmetric-power criterion characterizes the Yu
cells in types \(F_4\) and \(G_2\).
\end{corollary}

\begin{proof}
The propositions exclude every non-Yu balanced cell.  Conversely, Yu's
cells are unitary and hence positive on every \(W\)-type.
\end{proof}

\section{A second proof using boundary-induced \(W\)-types}
\label{sec:boundary-induced-proof}

The proof above is deliberately restricted to the reflection representation
and its second symmetric power.  We now give a shorter, independent
exclusion of the non-Yu cells when arbitrary finite Weyl-group types are
allowed.  We retain the first proof because its restriction to symmetric
degree two is important in applications where only petite types are
available.

The new argument has three ingredients: a parabolic-heredity lemma,
three small induced-sign calculations, and the root-poset classification of
locked ideals already used above.  In particular, it does not use a list
of relevant \(W\)-types.

\subsection{Parabolic induction of Hermitian forms}

Let \(J\subset\Delta\), let \(W_J\) be the corresponding parabolic Weyl
group, and let \(\tau\) be a unitary \(W_J\)-module.  Put
\[
                 \tau^W=\operatorname{Ind}_{W_J}^{W}\tau . \tag{15.1}
\]
If \(\nu\) is a generic real parameter, write \(\nu_J\) for its
restriction to the span of \(R_J\).

\begin{lemma}[parabolic heredity of indefiniteness]
\label{lem:boundary-parabolic-compression}
Assume that the ambient and Levi spherical standard modules are generic
and Hermitian.  In the compact induced model, the form on the
\(\tau^W\)-multiplicity space is congruent to
\[
 C_J(\nu)\otimes A^{R_J}_{\tau}(\nu_J),                    \tag{15.2}
\]
where \(C_J(\nu)\) is a nonsingular Hermitian form on the relative coset
space.  Consequently, if the Levi form is indefinite, then the ambient
form on \(\tau^W\) is indefinite.  The same statement holds for a
conjugate parabolic root subsystem.
\end{lemma}

\begin{proof}
Realize the spherical standard module by induction in stages from the
spherical standard module for the graded Hecke subalgebra \(\mathbb H_J\).
Its compact picture is the vector-space sum
\[
             \bigoplus_{x\in W^J}x\otimes X_J(\nu_J),       \tag{15.3}
\]
where \(W^J\) denotes the minimal coset representatives.  Factor the
normalized long operator into the Levi long operator and the normalized
relative operator.  The first gives the invariant form on
\(X_J(\nu_J)\).  The relative operator is an
\(\mathbb H_J\)-intertwiner.  Since the generic Levi standard module is
irreducible, Schur's lemma shows that, after a change of compact-picture
basis, it acts only on the relative coset factor.  Thus the induced
invariant form is the tensor product of the Levi form with a relative
Hermitian form \(C_J(\nu)\).  The latter is nonsingular away from the
ambient reducibility hyperplanes.  Applying unitary Frobenius reciprocity
to the \(\tau\)-multiplicity space gives (15.2).  This is the
compact-picture form of normalized induction in stages for graded affine
Hecke algebras; see also \cite{Lusztig}.

Diagonalize \(C_J\) by congruence.  Formula (15.2) becomes an orthogonal
sum of nonzero scalar multiples of the Levi form.  Every such multiple is
indefinite when the Levi form is indefinite, proving the assertion.
Conjugating the construction by an element of \(W\) proves the last
statement.
\end{proof}

If a Levi form merely fails to be positive, replace \(\tau\) by
\(\tau\oplus\mathbf1\).  The normalized form on the trivial type is
positive, so the enlarged Levi form is indefinite and the lemma applies.
Although \((\tau\oplus\mathbf1)^W\) need not be irreducible, its
indefinite multiplicity form implies that the form on at least one of
its irreducible \(W\)-constituents is not positive definite.

\begin{remark}
This lemma is precisely where allowing arbitrary \(W\)-types simplifies
the proof.  For a prescribed ambient type such as \(V\) or
\(\operatorname{Sym}^2V\), the required Levi type need not occur in the
correct induced multiplicity space and one needs the wall-recursion
arguments above.  For the induced type (15.1), the tensor factor occurs
by construction.  A literal principal compression to the identity coset
need not equal the Levi form; the congruence (15.2), rather than such a
principal-block assertion, is the required statement.
\end{remark}

\subsection{The two boundary seeds}

For the boundary antichains used below, the theorem of Sommers used in
Section~\ref{sec:reflection-signature}, or the explicit conjugacy
certificates below, shows that
\(W_A=\langle s_\alpha:\alpha\in A\rangle\) is conjugate to a parabolic
subgroup.  The natural boundary module is
\[
 M_A=\operatorname{Ind}_{W_A}^{W}\operatorname{sgn}_{W_A}
     \cong \operatorname{sgn}_W\otimes\mathbb R[W/W_A].    \tag{15.4}
\]
Thus every simple reflection acts in the coset basis as the negative of
a permutation matrix.  Formula (1.2) therefore makes every entry of its
form an explicitly checkable rational function.

\begin{lemma}[the \(B_2\) seed]\label{lem:boundary-b2-seed}
In type \(B_2\) or \(C_2\), let \(I\) be the ideal consisting of the two
nonsimple positive roots.  If \(\gamma=\min(I)\), the form on
\[
       M_\gamma=\operatorname{Ind}_{\langle s_\gamma\rangle}^{W(B_2)}
                    \operatorname{sgn}
\]
is indefinite on the cell of \(I\).
\end{lemma}

\begin{proof}
The module has dimension four.  Take both simple-root values equal to
\(3/5\).  In type \(B_2\) the root sequence has values
\[
                  \frac35,\frac65,\frac95,\frac35,
\]
and in type \(C_2\) the middle two values are interchanged.  In a coset
basis, exact multiplication of the four unnormalized factors in (1.2)
  gives the following principal minor of order two:
\[
 -\frac{56576}{390625}\quad(B_2),
 \qquad
 -\frac{206336}{390625}\quad(C_2).                         \tag{15.5}
\]
The form is therefore indefinite.  Both points lie in the required cell,
and the signature is constant on a generic cell.
\end{proof}

\begin{lemma}[the \(D_4\) seed]\label{lem:boundary-d4-seed}
Let \(I=R(D_4)^+\setminus\Delta(D_4)\).  Its minimal antichain is the
orthogonal triple
\[
 A=\{\alpha_c+\alpha_a,\alpha_c+\alpha_b,
                         \alpha_c+\alpha_d\},              \tag{15.6}
\]
where \(c\) is the trivalent node.  The form on
\[
 M_A=\operatorname{Ind}_{A_1^3}^{W(D_4)}\operatorname{sgn}
\]
is indefinite on the cell of \(I\).
\end{lemma}

\begin{proof}
The module has dimension \(192/8=24\).  At
\(\nu=\tfrac34\rho^\vee\), which is in the interior of this cell, exact
rational multiplication, after omission of the common positive
normalizing scalar, gives in one coset ordering the two diagonal entries
\[
 -\frac{2822935}{131072},
 \qquad
  \frac{197285}{131072}.                                  \tag{15.7}
\]
Thus the form is indefinite.  Exact symmetric elimination gives complete
signature \((n_+,n_-)=(10,14)\), although only the two entries in
(15.7) are needed.
\end{proof}

For later reference, the unique locked balanced \(G_2\) ideal is generated
by \(\alpha_1+\alpha_2\), with \(\alpha_1\) long.  The analogous module
\(\operatorname{Ind}_{A_1}^{W(G_2)}\operatorname{sgn}\) has dimension six.
At the point \((2/3,2/3)\), and again omitting the common positive
normalizing scalar, a principal minor of order two is
\[
                         -\frac{71125}{2187}.               \tag{15.8}
\]
It therefore supplies the \(G_2\) seed.

\subsection{Classical locked ideals}

We first treat type \(D_{2m}\).  Proposition
\ref{prop:d-locked-classification} says that its nonempty balanced locked
ideals are \(K_{n,j}\), \(1\le j<n/2\).  Their minimal antichains can be
written in coordinate form as
\begin{equation}
 \begin{split}
  &e_r-e_{r+2j}\quad(1\le r\le n-2j-1),\\
  &e_{n-2j}-e_n,\qquad e_{n-2j}+e_n .
 \end{split}
 \tag{15.9}
\end{equation}
Intersecting with the standard coordinate Levi on
\(e_3,\ldots,e_n\) gives
\[
                 K_{n,j}\cap R(D_{n-2})^+=K_{n-2,j}.       \tag{15.10}
\]
Iterating reaches \(K_{2j+2,j}\).  In that terminal ideal, use the
coordinate \(D_4\)-subsystem on the four coordinates
\[
                  e_1,\quad e_2,\quad e_{2j+1},\quad e_{2j+2}.
\]
A simple system is
\[
 e_1-e_2,\quad e_2-e_{2j+1},\quad
 e_{2j+1}-e_{2j+2},\quad e_{2j+1}+e_{2j+2}.                \tag{15.11}
\]
The first three nonsimple roots above its central node are exactly the
three roots in (15.9), and direct comparison gives
\[
 K_{2j+2,j}\cap R(D_4)^+
       =R(D_4)^+\setminus\Delta(D_4).                      \tag{15.12}
\]
Repeated application of Lemma~\ref{lem:boundary-parabolic-compression},
followed by Lemma~\ref{lem:boundary-d4-seed}, excludes every \(K_{n,j}\).

For types \(B_n,C_n\), put
\[
 a_{pq}=e_p-e_q,\qquad b_{pq}=e_p+e_q,
 \qquad z_p=e_p\ (B_n),\quad z_p=2e_p\ (C_n).
\]

\begin{lemma}[boundary \(B_2\) extraction]
\label{lem:boundary-bc-extraction}
Every nonempty balanced locked ideal in type \(B_n\) or \(C_n\) contains
indices \(p<q\) such that
\[
 I\cap\{a_{pq},z_q,b_{pq},z_p\}=\{b_{pq},z_p\}.            \tag{15.13}
\]
The four displayed roots form a parabolic coordinate subsystem of type
\(B_2\) or \(C_2\), and the right side is its bad ideal.
\end{lemma}

\begin{proof}
Use the shifted-staircase boundary scan in the proof of
Lemma~\ref{lem:bc-parity-descent}.  An ordinary inward corner is state D
of (12.6) and exposes a Yu pair, so it is impossible in a locked ideal.
At the shifted diagonal, terminal states A and B have, by direct
substitution in (12.5), an outer occupied pair \(b_{pq},z_p\) whose two
inner predecessors \(a_{pq},z_q\) are unoccupied.  This is (15.13).

In state C the complete outer strip is removed.  Equations (12.8) show
that the remaining ideal is exactly the intersection with the standard
coordinate subsystem of rank \(n-1\); the same cover check shows that it
is nonempty, balanced, and locked.  Apply induction on the rank.  A
witness in the smaller coordinate subsystem remains a witness in the
original ideal.  The rank-two endpoint is precisely (15.13), completing
the induction.
\end{proof}

Lemma~\ref{lem:boundary-b2-seed} and parabolic heredity now exclude every
nonempty balanced locked ideal in types \(B_n,C_n\).  This replaces both
the determinant split and the signed outer-strip descent needed when the
test types are restricted to symmetric degree two.

\subsection{Exceptional locked ideals}

The three nonempty balanced locked ideals of \(E_7\) have sizes
\(42,56,20\).  Each contains a parabolic \(D_4\) whose intersection is
the bad ideal of Lemma~\ref{lem:boundary-d4-seed}.  With the numbering
obtained from the \(E_8\) diagram by deleting node 7, one may take the
following centers and arms:
\[
\begin{array}{c|c|c}
|I|&c&(a,b,d)\\ \hline
42&(1,0,0,0,1,1,0)&(0,0,0,0,0,0,1),(0,0,1,0,0,0,0),(0,1,0,0,0,0,0)\\
56&(1,0,0,0,0,0,0)&(0,0,0,0,1,0,0),(0,0,1,0,0,0,0),(0,1,0,0,0,0,0)\\
20&(2,1,1,1,1,1,0)&(0,0,0,0,0,0,1),(0,0,0,0,1,0,0),(0,0,1,0,0,0,0).
\end{array}                                                 \tag{15.14}
\]
The Weyl words \(5465\), the identity, and
\(320140235465\), respectively, carry these bases to standard parabolic
\(D_4\) bases.  Hence Lemmas
\ref{lem:boundary-parabolic-compression} and
\ref{lem:boundary-d4-seed} exclude all three cells.

For \(E_8\), the four locked ideals of (8.2a) have sizes
\(96,64,112,32\).  The complete root-poset certificate gives respectively
\[
                              30,180,1,1                    \tag{15.15}
\]
bad \(D_4\) intersections.  For one witness in each row, the Weyl words
\begin{equation}
\begin{split}
 &054765,\\
 &102340120540120654023765,\\
 &\mathrm{id},\\
 &765401203240154023654765
\end{split}
\tag{15.16}
\end{equation}
carry its simple basis to the standard parabolic basis on nodes
\((0;4,2,1)\).  Thus all four are excluded by the same two lemmas.  Notice
that the three height-two roots of every witness belong to the minimal
boundary antichain; the induced detector is therefore genuinely attached
to the boundary.

Finally, every one of the seven locked \(F_4\) ideals in (14.1) contains
a parabolic bad \(B_2\).  A certificate is given in the following table;
the middle column lists a long and a short simple root for the \(B_2\),
in ambient simple-root coordinates, and the last column is a word carrying
the pair to the standard double edge.
\[
\begin{array}{c|c|c}
|I|&(a,b)&w\\ \hline
14&(1000,0110)&10\\
14&(0120,0001)&23\\
18&(1000,0110)&10\\
16&(1000,0110)&10\\
12&(1000,0111)&103\\
20&(0100,0010)&\mathrm{id}\\
4 &(0122,1110)&01321023.
\end{array}                                                 \tag{15.17}
\]
For each row, the two simple roots are outside \(I\), while
\(a+b,a+2b\) are inside.  Hence the intersection is the bad \(B_2\) ideal
and Lemma~\ref{lem:boundary-b2-seed} applies.  Equation (15.8) handles
the remaining type \(G_2\).

\subsection{Conclusion of the second proof}

\begin{theorem}[boundary-induced proof of exhaustion]
\label{thm:boundary-induced-exhaustion}
For every irreducible root system whose longest element is \(-1\), every
nonempty balanced locked ideal is indefinite on a finite \(W\)-module
obtained from the \(B_2\), \(D_4\), or \(G_2\) seed module by
parabolic induction.  Consequently its generic spherical unitary cells
are precisely Yu's cells.
\end{theorem}

\begin{proof}
The classical systems are covered by (15.9)--(15.13), the exceptional
systems by (15.14)--(15.17) and (15.8).  Starting with a balanced cell,
apply Yu moves until a locked endpoint is reached.  Positivity on every
fixed \(W\)-type is equivalent at the two ends of a Yu move.  A nonempty
endpoint is excluded by the boundary module just constructed, so a
unitary cell must reduce to the empty ideal.  Conversely Yu's theorem
makes every cell reducing to the empty ideal unitary.
\end{proof}

\begin{lemma}[Hermitian Levi reduction]
\label{lem:boundary-hermitian-levi}
Let \(\nu\) be a generic real Hermitian spherical parameter.  After Weyl
conjugacy there is a parabolic root subsystem
\(R_J=R_1\times\cdots\times R_t\), each of whose longest elements acts by
\(-1\), such that \(\nu\in V_J\).  The spherical standard module for
\(R\) is normalized parabolic induction from the tensor product of the
spherical standard modules with parameters \(\nu_i\) for the factors
\(R_i\).  It is unitary if and only if all those factor modules are
unitary.
\end{lemma}

\begin{proof}
Hermitianity gives \(w\nu=-\nu\) for some \(w\in W\).  After dividing by
the stabilizer of \(\nu\), we may take \(w\) to be an involution.
Richardson's theorem on involutions in finite Coxeter groups
\cite{Richardson} conjugates \(w\) to the longest element of a parabolic
subgroup on whose reflection space it acts by \(-1\).  Its \((-1)\)-space
contains \(\nu\), giving the asserted subsystem and its irreducible
factors.

Induction in stages identifies the ambient spherical standard module with
normalized parabolic induction from those factors and the zero, hence
unitary, character on the complementary split center.  Unitary inducing
data give a unitary induced module.  Conversely, the Barbasch--Moy
criterion supplies a finite type whose form is not positive for a
nonunitary generic Hermitian factor.  Adjoin the trivial type to make the
form indefinite, and apply the compact-picture factorization of
Lemma~\ref{lem:boundary-parabolic-compression}.  The induced \(W\)-type
makes the ambient form indefinite.  This proves the equivalence.
\end{proof}

\begin{corollary}[all root systems]
The conclusion of Theorem~\ref{thm:boundary-induced-exhaustion} holds for
every irreducible root system.  In particular, the second proof also
covers types \(A\), odd \(D\), and \(E_6\).
\end{corollary}

\begin{proof}
Apply the theorem to the irreducible factors supplied by
Lemma~\ref{lem:boundary-hermitian-levi}, and then induce back.  This is the
usual Hermitian Levi reduction and is compatible with the
downward-inheritance discussion of Section~\ref{sec:downward-inheritance}.
\end{proof}

\begin{remark}
The new proof isolates its case-dependent input very sharply.  Uniformly,
one uses reflection balance, Yu moves, parabolic heredity, and the three
seed calculations.  The finite part consists only of the locked-antichain
classification and the parabolic-subsystem certificates (15.14)--(15.17).
The much finer signature calculations on \(\operatorname{Sym}^2V\) remain
essential for the first proof, but are not needed here.
\end{remark}

\section{Consequences}

\begin{theorem}[symmetric-degree-two criterion]
Let \(\mathbb H(R)\) be the equal-parameter graded affine Hecke algebra
of a reduced crystallographic root system \(R\), and let \(\nu\) be a
generic Hermitian spherical parameter.  Then the spherical
\(\mathbb H(R)\)-module with parameter \(\nu\) is unitary if
and only if the normalized invariant forms are positive definite on the
reflection representation and on every irreducible constituent of
\(\operatorname{Sym}^2V\).
\end{theorem}

\begin{proof}
Unitarity implies positivity on every finite Weyl-group type.  The
algebra, its spherical module, and all the stated tests factor over the
irreducible components of \(R\), so it remains to prove the converse for
irreducible \(R\).  For the
converse, the simply-laced systems are covered by
Theorem~\ref{thm:simply-laced-converse}: even \(D\) is the uniform
inductive family, odd \(D\) and \(A\) follow by downward inheritance, and
\(E_6,E_7\) descend from the direct \(E_8\) assertion.  Types \(B\) and
\(C\) are covered by the sharper
Theorem~\ref{thm:bc-diagonal-converse}.  The two remaining irreducible
root systems, \(F_4\) and \(G_2\), are
Corollary~\ref{cor:exceptional-nonsimply-converse}.  In every case the
cells passing the stated tests are exactly Yu's cells, and Yu's theorem
makes those cells unitary.
\end{proof}

\begin{corollary}[field independence]\label{c:independence}
Let \(G\) be a simple Chevalley group over $\mathbb R$ or a nonarchimedean local field.  The generic
spherical unitary dual is independent of the local field, after the
standard normalization of real unramified parameters.  In a fundamental
Weyl chamber it is the disjoint union of \(2^\ell\) affine alcoves, with
alcoves understood relative to the Hermitian locus when
\(w_0\ne-1\), and
\[
       \ell=\nu\bigl(\Gamma(R^\vee_{\rm sh})\bigr)
       =\sum_{\Phi\subset R^\vee_{\rm sh}}
          \#\{e\in\operatorname{Exp}(\Phi):e<h_\Phi/2\}.
\]

If $G$ is simply-laced or $\mathrm{Sp}(2n)$, we can deduce the same description of the generic spherical unitary dual of $G(\mathbb C)$.
\end{corollary}

\begin{proof}
For nonarchimedean fields this is the spherical Barbasch--Moy/Lusztig
transfer to the equal-parameter graded affine Hecke algebra.  For
\(\mathbb R\), the Barbasch--Vogan petite-\(K\)-type comparison identifies
the relevant normalized operators with the same graded-Hecke operators
\cite{BarbaschPetite}.  The reflection representation and the required
constituents of its symmetric square lift to petite \(K\)-types of the
corresponding split real group.  When \(G(\mathbb C)\) is viewed as a real
group, the analogous role is played by small
\(G(\mathbb C)\)-representations.  For simply-laced \(G\), the required
Weyl-group representations occur in their zero-weight spaces; compare
\cite{Reeder}. For $\mathrm{Sp}(2n,\mathbb C)$, the stronger criterion in Theorem \ref{thm:bc-diagonal-converse} applies since $U_n$ is the $0$-weight space of a small representation.

The theorem and Proposition~\ref{prop:yu-matching-exponent} now give the assertion.
\end{proof}

\begin{remark}
The only simple group not covered by the argument in Corollary \ref{c:independence} is $\mathrm{So}(2n+1)$ because $U_n$ is not the $0$-weight space of a small representation in that case. 
\end{remark}

\begin{example}[The first failure of the \(V\)--\(X_n\) criterion]
\label{ex:C4-X-failure}
The reflection representation together with the off-diagonal constituent
\(X_n\) does not detect unitarity in the Hecke root type \(C_n\).  The smallest
counterexample occurs in \(C_4\).

Use the Bourbaki simple roots
\[
 \alpha_1=e_1-e_2,\qquad
 \alpha_2=e_2-e_3,\qquad
 \alpha_3=e_3-e_4,\qquad
 \alpha_4=2e_4,
\]
and take the dominant parameter
\[
                  \nu=\left(1,\frac57,\frac37,\frac17\right).
\]
No positive root has value \(1\), so \(\nu\) is generic.  Its root ideal is
\[
\begin{aligned}
 I(\nu)
  =\{&
     e_1+e_2,\ e_1+e_3,\ e_1+e_4,\ e_2+e_3,
     2e_1,\ 2e_2
    \}.
\end{aligned}
\]
Its minimal roots are
\[
             e_1+e_4,\qquad e_2+e_3,
\]
with respective simple-root coordinates
\[
             (1,1,1,1),\qquad (0,1,2,1).
\]
Thus
\[
 I(\nu)=
 \bigl\langle e_1+e_4,\ e_2+e_3\bigr\rangle_{\mathrm{up}}.
\]

The ideal \(I(\nu)\) is reflection-balanced and locked; it is not a Yu ideal.  Direct evaluation of the normalized
intertwining forms on
\[
 \operatorname{Sym}^2V=\mathbf1\oplus U_4\oplus X_4
\]
gives
\[
             n_-\bigl(A_{U_4}(\nu)\bigr)=2,
             \qquad
             n_-\bigl(A_{X_4}(\nu)\bigr)=0.
\]
Since the parameter is generic, the second equality means that the
\(X_4\)-form is positive definite.  The reflection form is also positive
definite because \(I(\nu)\) is balanced.  Nevertheless, the spherical
module is not unitary: its \(U_4\)-form is indefinite.

Since the Hecke root system in this paper is the coroot system of the
group, this type-\(C_4\) parameter corresponds to the complex group
\(\operatorname{SO}(9,\mathbb C)\).  It shows that the small
representations whose zero-weight spaces give the Weyl types \(V\) and
\(X_4\) do {\it not} suffice to detect spherical unitarity for
\(\operatorname{SO}(9,\mathbb C)\) (and recall that $U_4$ is not in the zero-weight space of a small $\operatorname{SO}(9,\mathbb C)$-representation).
\end{example}

\subsection{A two-dimensional picture of the \(E_8\) cells}
\label{subsec:e8-yu-projection}

Yu's labelled gallery gives a particularly economical planar picture of
the sixteen \(E_8\) alcoves.  Number them \(U_0,\ldots,U_{15}\), with
\(U_0\) the fundamental alcove and \(U_{15}=U_\infty\) the tail, and
number the two branches of the unique diamond \(U_7,U_8\).  Let \(c_i\)
be the barycenter of \(U_i\), in the Euclidean reflection representation
\(\mathfrak a\).  Define
\[
 e_1=\frac{c_{15}-c_0}{\lVert c_{15}-c_0\rVert},\qquad
 e'_2=(c_8-c_7)-\langle c_8-c_7,e_1\rangle e_1,\qquad
 e_2=\frac{e'_2}{\lVert e'_2\rVert},
\]
and put
\[
                         P_{\rm Yu}=\mathbb R e_1\oplus\mathbb R e_2.
\]
Thus the first coordinate records progress from the head to the tail of
Yu's gallery, while the second resolves its only branching.  Apart from
reversing one of the coordinate axes, the plane is determined entirely
by the head--tail symmetry and the diamond; it involves no choice of
coordinates for the \(E_8\) root system.

Figure~\ref{fig:e8-yu-head-tail} shows orthogonal projection onto
\(P_{\rm Yu}\).  The pale outlines are the actual projections of the
eight-dimensional alcoves.  Since projections of distinct alcoves may
overlap, the colored polygons are homothetic copies, with ratio \(0.28\),
about the projected barycenters.  They are display glyphs rather than a
second geometric projection.  The red segments reproduce the two-wall
adjacency graph.  In particular, the two points labelled \(7\) and \(8\)
display the two sides of the diamond, which merge again at \(U_9\).

\begin{figure}[ht]
  \centering
  \includegraphics[width=0.96\textwidth,trim=0 45 0 0,clip]
    {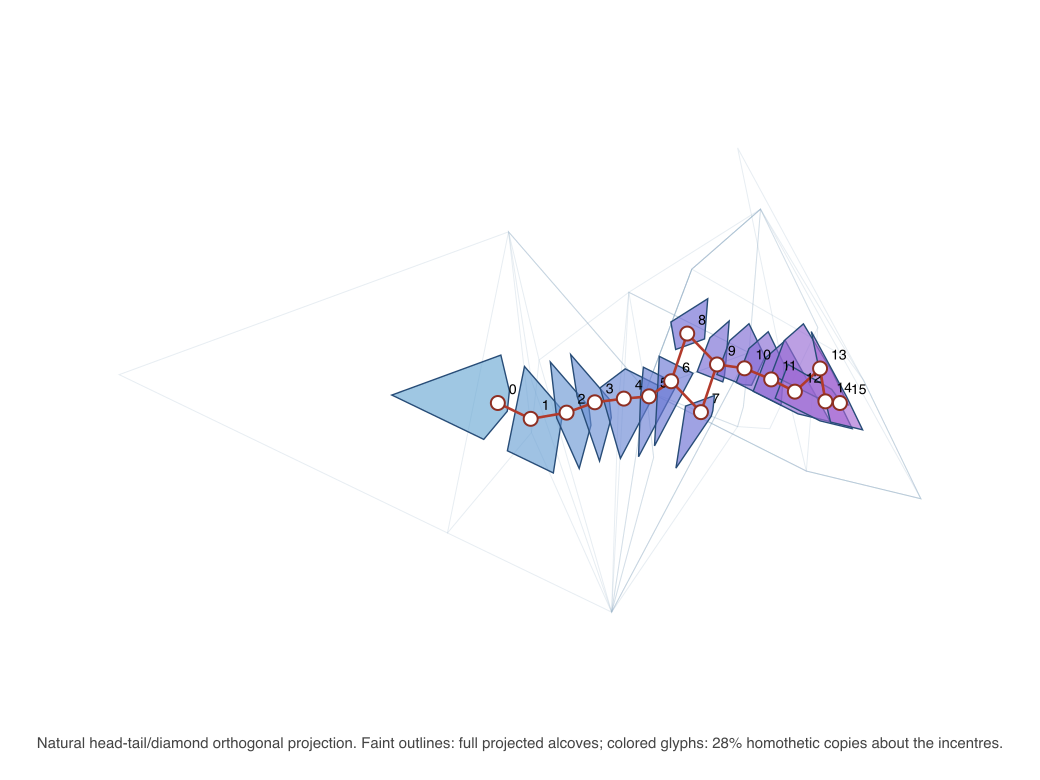}
  \caption{The sixteen \(E_8\) Yu alcoves in the natural
  head--tail plane \(P_{\rm Yu}\).}
  \label{fig:e8-yu-head-tail}
\end{figure}

\section*{Acknowledgements} The author used ChatGPT 5.6 Sol (Light) \cite{ChatGPTCodex} to
test ideas developed in this paper, particularly the combinatorial arguments involving the antichains,  and to assist with the preparation
and editing of the \LaTeX{} source.  The author checked the resulting
calculations, arguments, and text and takes full responsibility
for the contents of the paper. The author thanks Dragos Cri\c san for allowing him to include Theorem 2.2.

\end{document}